\documentclass[a4paper,fleqn, final]{cas-sc}

\usepackage[utf8]{inputenc} % set input eNoding (not needed with XeLaTeX)
\usepackage{amsmath}
\usepackage{amsfonts}
\usepackage{amssymb}
\usepackage{amsthm}
\usepackage{mathtools}
\usepackage{multirow}
\usepackage{tabularx}      
\usepackage{booktabs}      
\usepackage{array}         
\usepackage{eurosym}
\usepackage{booktabs}
\usepackage{url}
\usepackage{algorithm}
\usepackage{algpseudocode}
\usepackage{natbib}
\usepackage{pdfpages}

\restylefloat{table}
\usepackage{graphicx} % support the \iNludegraphics command and options
\usepackage[parfill]{parskip} % Activate to begin paragraphs with an empty line rather than an indent
\newtheorem{lemma}{Lemma}
\newtheorem{proposition}{Proposition}
\newtheorem{definition}{Definition}

\newtheorem{theorem}{Theorem}
\everymath{\displaystyle}

\everymath{\displaystyle}
\usepackage{xcolor}

\usepackage{amsmath}
\usepackage{amsfonts}
\usepackage{amssymb}

\usepackage[authormarkup=footnote,commandnameprefix=always]{changes}

\definechangesauthor[color=magenta]{Ruan}
\definechangesauthor[color=red]{Anthony}
\definechangesauthor[color=blue]{Mehdi}

\newcommand{\taglabel}[1]{\tag{#1}\label{eq:#1}}

\begin{document}

\let\WriteBookmarks\relax
\def\floatpagepagefraction{1}
\def\textpagefraction{.001}

\shorttitle{Models and Algorithms for Reserve Deliverability in Cross-Zonal Balancing Capacity Markets}    

\shortauthors{Mehdi Madani, Zejun Ruan, Anthony Papavasiliou}  

\title [mode = title]{Models and Algorithms for Reserve Deliverability in Cross-Zonal Balancing Capacity Markets}  




\author[1]{Mehdi Madani}%[<options>]

\cormark[1]


\ead{mma@n-side.com}


\credit{}

\cortext[1]{Corresponding author}

%\fntext[1]{The views and opinions expressed in this article  do not necessarily reflect those of N-SIDE.}
\affiliation[1]{organization={N-SIDE},
            addressline={Courbevoie 13}, 
            city={Louvain-la-Neuve},
%          citysep={}, % Uncomment if no comma needed between city and postcode
            postcode={1348}, 
            state={},
            country={Belgium}}

\author[2]{Zejun Ruan}%[]
\ead{zruan@mail.ntua.gr}

\author[2]{Anthony Papavasiliou}%[]

% Footnote of the second author

% Email id of the second author
\ead{papavasiliou@mail.ntua.gr}


\credit{}

\affiliation[2]{organization={National Technical University of Athens},
            addressline={Iroon Polytechneiou 9, Zografou}, 
            city={Athens},
%          citysep={}, % Uncomment if no comma needed between city and postcode
            postcode={157 72}, 
            state={},
            country={Greece}}

\begin{abstract}
In power markets where the physics of the transmission grid is closely represented in market clearing models, certain cross-zonal power exchanges can take place only if other exchanges occur concurrently. This leads to challenges in cross-zonal balancing capacity markets, where the activation of reserves in real time remains uncertain. Ensuring reserve deliverability in all activation scenarios is naturally modeled as a stochastic programming (SP) problem, whose size grows exponentially with the number of locations in the network. This formulation scales poorly in real-world applications. We first show that activation scenarios can be expressed in an order-book-agnostic way, reducing the challenge to a network modeling problem and improving computational performance. We then introduce a general inner approximation principle that we use to derive two scalable inner approximations and one column generation algorithm for tackling one of the approximations. The first inner approximation is well known to practitioners, and relates to a basic result for describing boxes in a polytope, while the second model, its associated column generation algorithm and finite-dimensional reformulation are based on semi-infinite linear programming and robust linear optimization. We compare the inner approximations to the exact SP formulation, which we also solve via a Dantzig-Wolfe decomposition for comparison purposes. Numerical results show that these inner approximations are much more scalable than the stochastic programming formulation, while reaping most of the benefits of cross-zonal exchanges. The approaches are of particular interest for future pan-European cross-zonal balancing capacity markets, and can also accommodate co-optimization of energy and balancing capacity products.
\end{abstract}


%{\let\newpage\relax\maketitle}
\maketitle

%\newpage
%\tableofcontents

\section{Introduction}

\subsection{Context and Contributions}

Reserves play a critical role in ensuring the safe and reliable operation of the transmission grid and security of supply. A key aspect of this responsibility is to maintain a real-time balance between electricity generation and consumption. To achieve this, System Operators rely on reserve commitments from generators or consumers to adapt their generation or consumption as needed. 

In Europe, reserve (also referred to as balancing capacity) markets are currently cleared separately from energy markets, and are mainly national markets. One exception is the Nordic region, see \cite{NbmHandbookCapacityMarket}, and the Baltic countries, see \cite{BBCM}, where cross-zonal balancing capacity markets are in place. In the upcoming years, a wider harmonization in the cross-zonal procurement of balancing capacity at a pan-European scale is expected. 

In this European context, reserve deliverability refers to a requirement introduced by European Transmission System Operators (see \cite{IIA, sdac_roadmap_study}). Reserve deliverability requires that System Operators must be able to  activate in real time any percentage of the balancing capacity that is procured cross-zonal in day-ahead markets, without violating network constraints.

Alternative reserve deliverability notions have been considered in US markets. In these markets, reserve requirements are often defined at a zonal or regional level, despite energy being cleared at nodal level. Reserve deliverability in these market clearing models must be secured in specific contingency scenarios. 
We refer the reader interested in the US practice to  \cite{chen2014incorporating}, which proposes ensuring reserve deliverability for specific post-contingency scenarios by explicitly imposing ``Post Zonal Reserve Deployment Transmission Constraints''.  A related earlier study \cite{zheng2008contingency} introduces a zonal reserve model derived from contingency event simulations in predefined reserve zones. Additionally, \cite{wu2021market} refines the approach in \cite{chen2014incorporating} and explores its market implications. 

Recently, \cite{vanCaelenberg2027constructing} has proposed to use Adaptive Robust Optimization to construct reserve deployment scenarios, improving real-time reserve deliverability. Further review of the literature and details on the enforcement of zonal reserve deliverability in U.S. markets are available in the cited references.

In this paper, we concentrate on an alternative formulation of reserve deliverability which does not concentrate on specific activation (deployment) scenarios but aims to accommodate \emph{any} possible configuration of reserve activation in real time. The requirement has been formulated based on our collaboration with European Transmission System Operators, with the aim of integrating this functionality in European market clearing algorithms for balancing capacity markets or for co-optimization in the Single Day-ahead Coupling.

This problem is referred to among European industry stakeholders as the ``deterministic reserve deliverability requirement" or more simply the ``deterministic requirement", in the sense that reserve deliverability should hold deterministically, in all scenarios, and not with a given probability, for a subset of activation scenarios. To the best of our knowledge, despite being intrinsically interesting and with practical applications to cross-zonal balancing capacity markets, this version of the problem has not been studied in the academic literature. On the other hand, it has received only a superficial treatment so far in the professional literature, see e.g., \cite{IIA, sdac_roadmap_study}.

An illustration of the reserve deliverability problem is provided in Figure \ref{fig:infeasible_matching}. The example assumes a single grid constraint, limiting to  $100 MW$ the capacity of the line $B-C$ in the direction $B$ to $C$, imposing the constraint $\frac{1}{3}NetPosition_A + \frac{2}{3}NetPosition_B \leq 100$. If the delivery of reserve were treated in a way as if its delivery would be required with certainty in real time, then the optimal matching of reserve bids in this example would consist of matching all the available supply (900MW). In such a matching, the 300 MW of reserve that is imported in $B$, which originates from $A$, frees up capacity on the line $B-C$ in the direction $B$ to $C$, since $\frac{1}{3}$ of the reserve exchange will flow from $A$ to $B$ through $C$. %In other words, way to see this is to note that $PTDF_A - PTDF_B=\frac{-1}{3}$. 
Such a flow would enable the market operator to import more reserve in location $C$ from location $A$ than would otherwise be possible: 600MW instead of 300MW. However, note that this matching and the network constraints are such that the reserve deliverability requirement is not satisfied. If the balancing capacity in location $B$ is eventually not activated, only 300 MW can be activated in location $C$ given that the relief from the exchange between location $A$ and location $B$ eventually does not materialize.

\begin{figure}%[H]
    \centering
    \includegraphics[width=0.6\linewidth]{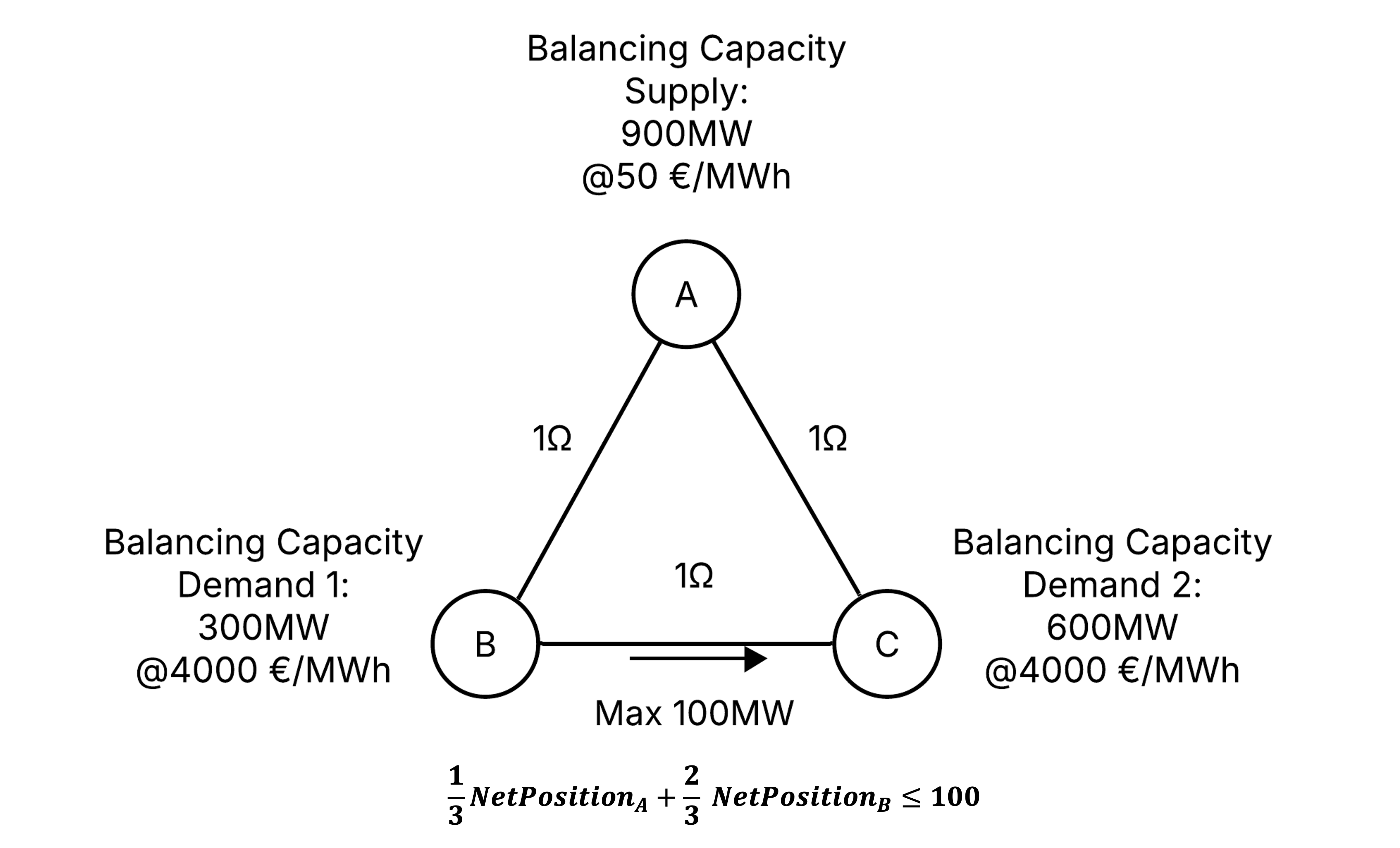}
    \caption{A three-zone example used for illustrating the problem of reserve deliverability. 
    Due to reserve deliverability constraints, it is \emph{not} possible to match the entire demand for reserve in locations B and C.} 
    \label{fig:infeasible_matching}
\end{figure}

This paper demonstrates how reserve deliverability can be efficiently enforced in linear approximations of meshed AC networks, without explicitly considering the exponential number of reserve activation scenarios. The results generalize to the case where power grids are modeled with general linear constraints on net injection variables.

The remainder of this introduction outlines the structure of the paper and summarizes the main results.

Section \ref{section:sp-formulation} formalizes the problem by providing the first ``most natural" formulation of the problem, as a stochastic linear programming formulation $InitSP$. We also introduce an alternative formulation $SP$, where the enforcement of the reserve deliverability requirement is made in an ``order-book agnostic way", where second-stage scenario-dependent decision variables only relate to the network model, and are constrained only by network-related first-stage decisions.

Section \ref{section: general-inner-approx-theorem} is central, as it introduces the general inner approximation principle at the core of the work. It allows us to generate inner approximations of $SP$ (and therefore of the reserve deliverability requirement), relying on a conceptual generalization of ad hoc practices known to practitioners to enforce the requirement in practice (which are recalled in Section \ref{section:bilateral-exchanges-formulation}). The level of generality is powerful enough to derive \emph{extensions} of these current practices, enabling to reach solutions with more welfare.

Section \ref{section:bilateral-exchanges-formulation} shows how the principle introduced in Section \ref{section: general-inner-approx-theorem} can be used to derive the current (non-exact) practice in the professional literature for enforcing reserve deliverability, see \cite{sdac_roadmap_study}, where it was first introduced (to the best of our knowledge), which we denote as $IBBE$. It is further discussed in \cite{madani2025inscribed}. 

Section \ref{section:multilateral-exchanges-formulation} further leverages the general principle introduced in Section \ref{section: general-inner-approx-theorem} to extend the $IBBE$ formulation. This leads to a new semi-infinite linear programming formulation, which we denote as $IBME$. We develop tractable computational approaches for tackling this formulation in the remaining sections.

Section \ref{section:finite-reformulation} provides a finite-dimensional reformulation of $IBME$, denoted as $FIBME$. It is derived by leveraging robust linear optimization techniques described for instance in \cite{sun2021robust}.

In Section \ref{section:cg-algo}, we tackle the semi-infinite linear programming formulation $IBME$ by a column generation algorithm. We further prove finite convergence of the proposed algorithm.

Section \ref{section:numerical-experiments} presents numerical experiments. We compare the model currently considered by Transmission System Operators, to the extensions $IBME$ and $FIBME$, and to the Dantzig-Wolfe decomposition method directly applied to the reformulation $SP$.

The results show that the inner approximations are much more computationally-efficient than $SP$, while having minimal impact on solution quality. A tradeoff is revealed between speed and accuracy of the approximations. The column generation approach solving the semi-infinite linear programming formulation $IBME$ clearly stands out, as it combines high accuracy and high performance.

We conclude in Section \ref{section:conclusions}.

\section{Reserve Deliverability as a Stochastic Linear Program} \label{section:sp-formulation}

We commence in subsection \ref{subsection:base-model} by describing the market clearing problem without considering reserve deliverability. 

Subsection \ref{subsection:InitSP} then introduces the first exact formulation $InitSP$ of the reserve deliverability requirement. The formulation fully hides the energy part of the model. Concretely, terms related to energy are moved to the 
right-hand sides of network and generator capacity constraints, and we focus on a formulation in terms of reserve variables alone in the sequel. 
The formulation directly  considers the demand activation scenarios, together with the ``best response" (recourse decisions) in terms of generator activations. This natural formulation suffers from two drawbacks that we address in this same section by providing a non-trivial reformulation of the problem. First, it scales exponentially with the number of demands and generators, since we represent the demand and generator activation variables for each extreme scenario of demand activation. Second, the formulation may not make clear that the reserve deliverability problem for the exchange of reserves is fundamentally a network modeling problem. 

We then introduce in subsection \ref{subsection:SP} the key reformulation $SP$ where scenario-dependent decisions only concern network variables and are only constrained by network-related first-stage decisions. Along the way, we also state key properties of the models used to prove the validity of reformulation $SP$.

\subsection{Base Model, Notation and Terminology} \label{subsection:base-model}

We first present the base model that co-optimizes energy and reserves without the reserve deliverability requirement. We then specialize it to a reserve-only setting, reflecting the paper's focus on balancing capacity markets. 

This section also introduces the notation and terminology used throughout the paper.

The base model is formulated as follows:
\begin{align}
& \max_{p,d,pr,dr, ne_z, nr_z} \ \ \sum_l V^l d_l + \sum_r VR^r dr_r - \sum_{g}^{}C^{g}p_{g}  \taglabel{Base1} \\
& d_l \leq D_l                  & \forall l \in Loads                                  \taglabel{Base2} \\
& dr_r \leq DR_r                & \forall r \in RDB                                  \taglabel{Base3}\\
& pr_{g} \leq P_g  - p_g       & \forall g \in G                                       \taglabel{Base4}\\
&  \sum_{l | z(l) = z} d_l - \sum_{g | z(g) = z} p_g = -ne_z & \forall z \in N           \taglabel{Base5}\\
&  \sum_{r | z(r) = z} dr_r - \sum_{g | z(g) = z} pr_g = -nr_z  & \forall z \in N        \taglabel{Base6}\\
& & \notag \\
& \sum_z ne_z = 0                                                                 \taglabel{Base7}\\
& \sum_z nr_z = 0                                                                \taglabel{Base8} \\
& \sum_{z \in  N}^{} PTDF_{k,z} nr_{z}  \leq \ F_{k}^{\max} - \sum_{z \in N}^{} PTDF_{k,z} ne_z& \forall k \in CNE \taglabel{Base9} \\
& d, dr, p, pr \geq 0 \taglabel{Base10}
\end{align}

In this base model, the expression \eqref{eq:Base1} is the welfare objective function with the utility of the accepted demands given by $\sum_l V^l d_l + \sum_r VR^r dr_r$ and system generation costs given by $ \sum_{g}^{}C^{g} p_g$. Here, $V^l$ represents the marginal value of the energy demand indexed by $l$ in the set of loads $Loads$, $VR^r$ the marginal value of the reserve demand indexed by $r$ in the set of reserve demand bids $RDB$, $d_l$ the volume of matched energy demand, $dr_r$ the matched reserve demand, $C^{g}$ the marginal cost of generator $g$, with $g$ in the set of generators $G$, and $p_g$ its energy power output. 

The constraints \eqref{eq:Base2}-\eqref{eq:Base4} are demand and generation capacity constraints, where $pr_g$ is the balancing capacity provided by generator $g$. 

Constraints \eqref{eq:Base5}-\eqref{eq:Base6} are respectively the energy and reserves locational power balance conditions, relating for each zone $z$ in the set of zones $N$, the net positions - respectively $ne_z$, $nr_z$ - to supply and demand. 

Constraints \eqref{eq:Base7}-\eqref{eq:Base8} enforce global power balances for energy and reserves, while \eqref{eq:Base9} are flow-based transmission constraints, one per critical network element $k$ in the set of critical network elements $CNE$, with  the consumption of cross-zonal capacity by the energy exchanges already moved to the right-hand side as an offset of the capacities $F_{k}^{\max}$. The coefficients $PTDF$ are standard power transfer distribution factors.

In the rest of the paper \emph{after this subsection}, we will focus on secure cross-zonal exchanges of balancing capacity, and consider the energy part as fixed. This means that we remove the variables $p,d, ne$ from the model and redefine the right-hand side of \eqref{eq:Base4}, \eqref{eq:Base9} accordingly, by setting $\widetilde{P_g} := P_g - p_g$ and $\widetilde{F_{k}^{\max}} := F_{k}^{\max}  - \sum_{z \in N}^{} PTDF_{k,z} ne_z$. However, all the developments which follow could be adapted to fit in a context where energy and balancing capacity are co-optimized.

\subsection{Initial SP Formulation - InitSP} \label{subsection:InitSP}

Let us consider the following first-stage day-ahead decisions, the description of which has been given above: $p_g, pr_g, d_l$, $dr_r$ (generation, reserve provision, energy demand and reserve demand). These decisions determine day-ahead first-stage energy net injections $ne_z$ and reserve net positions $nr_z$.  Note that in the following, the energy-related decision variables $p, d, ne$ are considered fixed.

These first-stage day-ahead decisions should be such that, for each reserve activation scenario $s$, specifying (second-stage) activation values $drAct_{r,s}$ of the reserve demands $dr_r$ (with $0 \leq drAct_{r,s} \leq dr_r$), there exist recourse (second-stage) upward generation decisions $prAct_{g,s} \leq pr_g$ leading to second-stage net injection values $nrAct_{z, s}$ that still satisfy network constraints. Note that all second-stage decisions are  indexed by $s$. 

We provide an initial stochastic programming formulation, denoted $InitSP$, described as follows:
\begin{align}
& \max_{pr,dr, nr_z} \ \ \sum_r VR^r dr_r   \taglabel{InitSP1} \\
& dr_r \leq DR_r                & \forall r \in RDB                                  \taglabel{InitSP3}\\
&  pr_{g} \leq \widetilde{P_g}         & \forall g \in G                                       \taglabel{InitSP4}\\
&  \sum_{r | z(r) = z} dr_r - \sum_{g | z(g) = z} pr_g = -nr_z  & \forall z \in N        \taglabel{InitSP6}\\
& & \notag \\
& \sum_z nr_z = 0                                                                \taglabel{InitSP8}\\
& & \notag \\
&\forall s=(s_1, ..., s_r, ..., s_{|RDB|}) \in [0;1]^{|RDB|}: & \notag \\
& & \notag \\
& drAct_{r,s} = s_r \ dr_{r}  & \forall r\in RDB                     \taglabel{InitSP10} \\
& prAct_{g,s} \leq pr_{g}   & \forall g\in G                     \taglabel{InitSP11} \\
&  \sum_{r | z(r) = z} drAct_{r,s} - \sum_{g | z(g) = z} prAct_{g,s} + nrAct_{z,s} = 0 & \forall z\in N \taglabel{InitSP12} \\
& \sum_z nrAct_{z,s} = 0 &  \taglabel{InitSP13} \\
& \sum_{z \in N}^{} PTDF_{k,z} nrAct_{z,s}  \leq \widetilde{F_{k}^{\max}} & \forall k \in CNE \taglabel{InitSP14} \\
&dr, pr, drAct, prAct \geq 0 \taglabel{InitSP15}
\end{align}

The feasible set of the stochastic formulation $InitSP$ therefore requires that the first-stage decisions $pr,dr$ and the implied net positions $nr$ be such that any possible activation scenario \\ $s=(s_1, ..., s_r, ..., s_{|RDB|}) \in [0,1]^{|RDB|}$ of the reserve demands $dr$ can be managed by the network (where $|RDB|$ denotes the number of ``Operating Reserve Demand Curve" orders, i.e., reserve demand bids).  Note our specific choice of scenario set. It is intended to describe all possible ways in which matched reserve demands could be activated, as a percentage of the quantity of reserve demand matched in the first stage of the problem.

The set of scenarios $[0;1]^{|RDB|}$ is an infinite and uncountable set of scenarios of reserve activations which maps all possible combinations of activation percentage $s[r] \in [0,1]$ of a reserve demand bid $r \in RDB$.

We will show in subsection \ref{subsection:SP} that the reserve deliverability problem can be reformulated in a way that decouples the description of the feasible exchanges of balancing capacity from the bid acceptance variables and power balance constraints. This can be achieved by  reworking constraints \eqref{eq:InitSP10}-\eqref{eq:InitSP12} in which scenarios are defined in terms of second-stage reserve demand bid activations. Instead, we propose scenarios that are purely described in terms of variables that are used to describe network constraints.

\subsection{SP Formulation (Order-book Agnostic Reformulation of $InitSP$)} \label{subsection:SP}

Compared to $InitSP$, the idea underlying the formulation $SP$ provided below is to describe the set of feasible exchanges of balancing capacity using a variable $TRD_z$.  This variable represents the \emph{potential} total matched reserve demands in location $z$ \emph{independently} from what is actually present in the order books. This total reserve demand may or may not be activated in the second stage, and the amount activated in a given scenario $s$ is denoted as $TRDAct_{z,s}$. The same logic applies to the variable $TRS_{z}$. The activation  $TRSAct_{z,s}$ that actually occurs in the second stage must be below the initial amount $TRS_{z}$. Our intent with this formulation is to isolate the model of the deterministic requirement,  
see \eqref{eq:SP6}. The complete description of the set $XSP$ is given by \eqref{eq:XSP1}-\eqref{eq:XSP9}. 

This reformulation is useful for two reasons. First, it helps to stress that the reserve deliverability problem is a network modeling problem, where the question is how to adequately represent the set of energy and reserve exchanges satisfying specific properties. This requirement can be expressed in a way which is agnostic to the order book that is submitted to the market clearing algorithm, and has to do purely with the network model. Second, from a performance point of view, it helps to reduce the number of second-stage decision variables. The second-stage reserve bid acceptances are removed and replaced by second-stage net positions and total demand or supply variables which are less numerous.

Concretely, we propose the following stochastic programming model for representing cross-zonal exchanges, which we denote SP:
\begin{align}
& \max_{dr,pr, nr} \ \ \sum_r VR^r dr_r  \taglabel{SP1} \\
& dr_r \leq DR_r                & \forall r \in RDB                                  \taglabel{SP2}\\
& pr_{g} \leq \widetilde{P_g}         & \forall g \in G                                       \taglabel{SP3}\\
&  \sum_{r | z(r) = z} dr_r - \sum_{g | z(g) = z} pr_g = -nr_z  & \forall z \in N        \taglabel{SP4}\\
& & \notag \\
& dr, pr \geq 0                                                     \taglabel{SP5}\\
& nr \in Proj_{nr}(XSP)                                                        \taglabel{SP6}
\end{align}

The set $Proj_{nr}(XSP)$ corresponds to the set of feasible exchanges of reserves. It is given by the projection on the space of the variables $nr$ of the set described by:
\begin{align}
& \sum_z nr_{z} = 0 & (First\ Stage) \taglabel{XSP1} \\
& \sum_{z \in N}^{} PTDF_{k,z} nr_{z}  \leq \ \widetilde{F_{k}^{\max}}& \forall k \in CNE \ (First\ Stage)  \taglabel{XSP2} \\
& nr_{z} = TRS_{z} - TRD_{z} & \forall z \in N \ (First\ Stage) \taglabel{XSP3} \\
& & \notag \\
&\forall s=(s_1, ..., s_n, ..., s_{|N|}) \in \{0;1\}^{|N|}:
& & \notag
\end{align}

\vspace{-0.8cm}

\begin{align} 
& TRDAct_{z,s} = s_z \ TRD_{z} &\forall z \in N \ (Second\ Stage) \taglabel{XSP4} \\
& TRSAct_{z,s} \leq TRS_{z} &\forall z \in N \  (Second\ Stage)  \taglabel{XSP5} \\
&  nrAct_{z,s} = TRSAct_{z,s} - TRDAct_{z,s} &\forall z \in N \  (Second\ Stage)  \taglabel{XSP6} \\
& \sum_z nrAct_{z,s} = 0 &\ (Second\ Stage) \taglabel{XSP7} \\
& \sum_{z\in N}^{} PTDF_{k,z} nrAct_{z,s}  \leq \ \widetilde{F_{k}^{\max}} & \forall k \in CNE\ (Second\ Stage) \taglabel{XSP8} \\
& TRS, TRD, TRSAct, TRDAct \geq 0  & \taglabel{XSP9}
\end{align}

We show below in Proposition \ref{proposition-initsp-sp} that $InitSP$ and $SP$ are equivalent in the sense that they lead to the same welfare and admit the same feasible first-stage decisions $dr, pr, nr$, i.e. $Proj_{dr, pr, nr}(InitSP) = Proj_{dr, pr, nr}(SP)$.

Before demonstrating this result, it is needed to cover a few key properties of $XSP$. These include a proposition on extreme scenarios holding both for $SP$ and $InitSP$ (Proposition \ref{prop:xsp} below), and two key properties of $SP$.

First, $SP$ relies on a finite set of scenarios $S_2 = \{s | s=(s_1, ..., s_z, ..., s_{|N|}) \in  \{0;1\}^{|N|} \}$ while $InitSP$ relies on the uncountable scenarios $S_{InitSP} = \{s | s=(s_1, ..., s_r, ..., s_{|RDB|}) \in [0;1]^{|RDB|} \}$. The first proposition below shows that considering only the finite set of extreme scenarios $S_2$ with either full activation of all the reserve demand in zone $z$ (corresponding to $s_z = 1$), or no activation at all (corresponding to $s_z=0$) is actually equivalent to considering all extreme and non-extreme scenarios in $S_1$, where the set $S_1$ is defined below. This proposition holds both for $InitSP$ and for $SP$. We prove it for the latter, for which the notation is simpler.

\begin{proposition} \label{prop:finite-set-scenarios} \label{prop:xsp}
    In the constraints \eqref{eq:XSP4}-\eqref{eq:XSP9}, the set of feasible first-stage decisions \\ $(nr, \ TRD, TRS)$ is the same whether $S_1$ or $S_2$ is used, where $S_1 = \{s | s=(s_1, ..., s_z, ..., s_{|N|}) \in [0;1]^{|N|} \}$, and \\ $S_2 = \{s | s=(s_1, ..., s_z, ..., s_{|N|}) \in \{0;1\}^{|N|} \}$.
    
    More formally, let us denote $XSP(S)$ the set $XSP$ parametrized by some scenario set $S$ defining an activation value $s_z$ per zone for each scenario $s \in S$ (a scenario being a vector of activation values). For $S_1$ and $S_2$ defined above, we have that $Proj_{(nr, \ TRD, TRS)}(XSP(S1)) = Proj_{(nr, \ TRD, TRS)}(XSP(S2))$.
\end{proposition}

\begin{proof}
    Let us consider the scenario-dependent constraints \eqref{eq:XSP4}-\eqref{eq:XSP9} as inequalities in the second-stage scenario-dependent variables $TRDAct_{z,s}, TRSAct_{z,s}, nrAct_{z,s}$. For each scenario $s$ in a scenario set $S$, the right-hand side vector $y_s, s \in S$ of these inequalities is therefore $(s_z TRD_z, TRS_z, 0, 0, \widetilde{F_{k}^{\max}}, 0)$.

     We now use Lemma 1 in \cite{geoffrion} stating the following. Let $g$ be a convex function and let $Y := \{ y \in \mathcal{R}^m: g(x) \leq y \texttt{\ for some\ } x\}$; then $Y$ is a convex set.
    
    Therefore, the set of right-hand sides $y$ for which these inequalities admit a feasible solution is a convex set. The result then follows directly from the fact that for $S_1, S_2$ as defined above, $\{y_s, s \in S_1 \} = conv(\{y_s, s \in S_2 \})$. Concretely, if the inequalities are feasible for each $y_s, s \in S_2$, as imposed by the constraints of $SP$ if the scenario set $S_2$ is considered, they will also be feasible for each convex combination of these right-hand sides. These convex combinations exactly correspond to the right-hand sides obtained if $S_1$ is considered. Reciprocally, if these right-hand sides are feasible for all choices of values in $S_1$, they are also trivially feasible for all choices of values in $S_2$, since $S_2 \subseteq S_1$.
\end{proof}

We now derive two Lemmas which describe two key properties of $SP$. These properties are used later on, notably to show the equivalence between $InitSP$ and $SP$.

Lemma \ref{lemma-netting_total_demand_supply} below asserts that the values of $TRD_z$ and $TRS_z$ can be ``netted''. What we mean by this is that, whatever the values that these variables take which define $nr_z$, alternative values where at most one of them is non-zero, and corresponding second-stage decisions $TRDAct, TRSAct, nrAct$, can be defined that (a) leave $nr_z$ unchanged, and (b) such that the new point still belongs to $XSP$.

Lemma \ref{lemma-shifting_total_demand_supply} states that, at any zone $z$, we can add a constant $k$ to the variables representing the local supply $TRS^*_z$ and local demand $TRD^*_z$, and adapt accordingly the second-stage decisions $TRDAct, TRSAct$, such that (a) $nr_z, nrAct_z$ are left unchanged, and (b) such that the new point still belongs to $XSP$.

\begin{lemma}[Key SP Property 1 - Netting Lemma]  \label{lemma-netting_total_demand_supply}

    Let $(nr^*, TRD^*, TRS^*)$ \\ $\in Proj_{(nr, TRD, TRS)}(XSP) $, and $z$ be any location, and consider the following netted values for $TRD_z, TRS_z$ which leave the zonal net position $nr_z$ unchanged: 
    
    \begin{align*}
       & TRD_z^\#:= -\min(nr^*_z,0) \taglabel{XSP-Netting1} \\
       & TRS_z^\#:= \max(nr^*_z,0) \taglabel{XSP-Netting2}  \\
       & TRD_l^\#:= TRD^*_l & \forall l \neq z \taglabel{XSP-Netting3}  \\
       & TRS_l^\#:= TRS^*_l & \forall l \neq z \taglabel{XSP-Netting4} 
    \end{align*}

    Then $(nr^{\#}, TRD^{\#}, TRS^{\#})$ also belongs to $Proj_{(nr, TRD, TRS)}(XSP)$, i.e. can be complemented with values for the scenario-dependent variables  $TRDAct^{\#}_{z,s}$, $TRSAct^{\#}_{z,s}$, $nrAct^{\#}_{z,s}$ to obtain a feasible point of $XSP$.

\end{lemma}

\begin{proof}
    See Appendix \ref{appendix:proofs}
\end{proof}

\begin{lemma}[Key SP Property 2 - Shifting Lemma]  \label{lemma-shifting_total_demand_supply}
    Let $(nr^*, TRD^*, TRS^*) \in Proj_{(nr, TRD, TRS)}(XSP) $, $z$ be any zone, and consider the following shifted values for $TRD_z, TRS_z$ which leave the zonal net position $nr_z$ unchanged: 
    
    \begin{align*}
       & TRD_z^\#:= TRD^*_z + K \taglabel{XSP-Shifting1} \\
       & TRS_z^\#:= TRS^*_z + K \taglabel{XSP-Shifting2}  \\
       & TRD_l^\#:= TRD^*_l & \forall l \neq z \taglabel{XSP-Shifting3}  \\
       & TRS_l^\#:= TRS^*_l & \forall l \neq z \taglabel{XSP-Shifting4} 
    \end{align*}

    where $TRD^*_z$ and $TRS^*_z$ have been increased by the same constant $K$. Then $(nr^*, TRD^{\#}, TRS^{\#})$ also belongs to $Proj_{(nr, TRD, TRS)}(XSP)$, i.e. can be complemented with values for the scenario-dependent variables  $TRDAct^{\#}_{z,s}$, $TRSAct^{\#}_{z,s}$, $nrAct^*_{z,s}$ to obtain a feasible point of $XSP$.

\end{lemma}

\begin{proof}
    See Appendix \ref{appendix:proofs}
\end{proof}

\begin{proposition} \label{proposition-initsp-sp}
    $Proj_{dr, pr, nr}(InitSP) = Proj_{dr, pr, nr}(SP)$. In other words, both $InitSP$ and $SP$ admit the same set of key first-stage decisions $dr, pr, nr$, i.e. day-ahead bid acceptances and cross-zonal exchanges.
\end{proposition}

\begin{proof}
    
See Appendix \ref{appendix:proofs} \end{proof}

The main challenge with the formulation $SP$ is that its size grows exponentially with the number of zones in the network. Even if it already improves on $InitSP$ whose size grows exponentially with respect to reserve demands that can be activated, and which contains decision variables for generators indexed by scenarios, it still doesn't scale well when moving to large-scale instances.

The next sections propose computationally efficient inner approximations, where the number of variables and constraints grow linearly in the number of zones. This allows us to enforce reserve deliverability in a scalable way, with very low impacts on welfare or system costs.

Before moving to these approaches, it is important to acknowledge that the $SP$ formulation above could be solved via decomposition methods, such as a Dantzig-Wolfe column generation algorithm, where we iteratively refine an inner approximation of $Proj_{nr}(XSP)$. This decomposition approach is numerically tested in Section \ref{section:numerical-experiments}. However, it doesn't remove the burden of the exponential number of scenarios which still appear in the pricing problems.

\section{A General Inner Approximation Theorem}\label{section: general-inner-approx-theorem}

In the present section, we introduce a general result for building inner approximations of $XSP$. These inner approximations will turn out to be powerful for deriving both models and algorithms based on column generation for solving $SP$. % to any desired precision. 
We will see that the current approach for modeling reserve deliverability in European studies on cross-zonal capacity, discussed in Section \ref{section:bilateral-exchanges-formulation}, is a special case of this general result.

While we focus on upward balancing capacity in the remainder of this paper, this general result can be adapted to the case of both upward and downward balancing capacity, such that standard models in the professional literature adapting the inner approximation in Section \ref{section:bilateral-exchanges-formulation} to both upward and downward balancing capacity procurement can be recovered as special cases. See \cite{PapavasiliouAvila2024} for these special cases, and the Supplementary Material on how to derive them from Theorem \ref{theorem-key-inner-approximation-theorem} adapted to upward and downward balancing capacity setups. 

The inner approximation theorem below first requires the definition of parameters that will play a key role.

\medskip

\begin{definition}[Secure Cross-zonal Capacities $W_k$] \label{def:secure-czc}
    Consider a cross-zonal exchange $(nr_1, ..., nr_{|N|})$ satisfying the initial global balance constraints \eqref{eq:XSP1} and consider a set of parameters $W_k, k\in CNE$.
    
    Let us denote by $XSP[F_{k}^{\max} \rightarrow W_k]$ the set $XSP$ obtained if $\widetilde{F_{k}^{\max}}$ is replaced by $W_k$ for each $k \in CNE$ in \eqref{eq:XSP2} and \eqref{eq:XSP8}.
    
    The parameters $W_k$ are called \emph{Secure Cross-zonal Capacities for $(nr_1, ..., nr_{|N|})$} if $(nr_1, ..., nr_{|N|}) \in Proj_{nr}(XSP[F_{k}^{\max} \rightarrow W_k])$, or in other words, if $(nr_1, ..., nr_{|N|})$ is feasible for \eqref{eq:XSP1}-\eqref{eq:XSP9} if the capacities $\widetilde{F_{k}^{\max}}$ are replaced by $W_k$.
\end{definition}

We now move to the statement of the central theorem of the paper. The result can be described in a quite intuitive way as follows. A certificate that given cross-zonal exchanges of balancing capacity satisfy the reserve deliverability requirement can be provided by showing that these exchanges can be decomposed into a specific constrained sum of cross-zonal exchanges. The constraint on the sum of these cross-zonal exchanges requires that cross-zonal capacities that would make these exchanges ``reserve deliverable" do not sum up to a value exceeding available transmission capacities (per critical network element).  The idea for proving the result relies on showing that, for any activation scenario where some reserve demands may eventually not be activated in specific locations in real time, one can simply reduce the amount of cross-zonal exchanges appearing in the constrained sum, while not exceeding the available cross-zonal capacities.

\begin{theorem}\label{theorem-key-inner-approximation-theorem}
    Consider a finite number of cross-zonal exchanges $(nr^i_1, ..., nr^i_{|N|}), i \in \{ 1, ..., T\}$ satisfying the initial global balance constraints \eqref{eq:XSP1}, and for each of these exchanges, consider a set of Secure Cross-zonal Capacities $W^i_k$ as defined in Definition \ref{def:secure-czc}.

    Then any %cross-zonal exchange 
    solution to the following set of inequalities in the variables $nr, f_i$, considering the $nr^i$ as fixed parameters, will be a secure cross-zonal exchange satisfying the reserve deliverability requirement, i.e. will belong to $Proj_{nr}(XSP)$:
%\textbf{INAP}

    \begin{align}
& \begin{pmatrix}
nr_{1} \\
\vdots \\
nr_{z} \\
\vdots \\
nr_{|N|}
\end{pmatrix}
=\sum_{i= 1, ..., T} \begin{pmatrix} 
nr^i_{1} \\
\vdots \\
nr^i_{z} \\
\vdots \\
nr^i_{|N|}
\end{pmatrix} f_i \taglabel{INAP-1} \\
& \sum_{i= 1, ..., T}  W^i_k \ f_i \leq \widetilde{F_k^{max}} & \forall k \in CNE \taglabel{INAP-2}\\
& f \geq 0 \taglabel{INAP-3}
\end{align}
\end{theorem}

\begin{proof}
    See Appendix \ref{appendix:proofs}
\end{proof}

Note that inner approximations built using Theorem \ref{theorem-key-inner-approximation-theorem} can be exact. Suppose that $(nr_1^*, ..., nr_{|N|}^*)$ is part of an optimal solution to $SP$. Then Theorem \ref{theorem-key-inner-approximation-theorem} provides an exact inner approximation if this optimal solution is one of the exchange patterns indexed by $i$ in the theorem, with $W_k^* := \widetilde{F_k^{max}}, \ \forall k \in CNE$ (since the point belongs to $XSP$, the original capacities $\widetilde{F_k^{max}}$ are secure cross-zonal capacities in the sense of Definition \ref{def:secure-czc}).

Also, advanced column generation algorithms leveraging Theorem \ref{theorem-key-inner-approximation-theorem} can be designed. Indeed, for any vector of cross-zonal exchanges $(nr_1, ..., nr_{|N|})$, one can determine corresponding (non-unique) secure capacities $W_k$ and then add this vector to the list of secure exchange patterns used to build the inner approximation. Dedicated column generation algorithms should therefore focus on which exchange patterns $(nr_1, ..., nr_{|N|})$ to iteratively add for building a sequence of inner approximations converging to an optimal solution to $SP$. Exact approaches that can be developed in that spirit, although beyond the scope of the present work, are shortly discussed in Section \ref{section:cg-algo}.

In the rest of the paper, we focus on specific approximation models (see Sections \ref{section:bilateral-exchanges-formulation}, \ref{section:multilateral-exchanges-formulation}, \ref{section:finite-reformulation}), or custom column generation %-like 
algorithms (see Section \ref{section:cg-algo}), which are easy to derive from Theorem \ref{theorem-key-inner-approximation-theorem}, and which therefore enforce the reserve deliverability requirement. Although they enforce the requirement in a conservative manner, numerical experiments in Section \ref{section:numerical-experiments} show that they reach near-optimal solutions while being much more scalable than directly solving $SP$.

\section{Inner approximation of reserve deliverability by inscribed boxes in the space of bilateral exchanges}
\label{section:bilateral-exchanges-formulation}
%\label{section:net-injections-formulation}

The model derived in this section, referred to as $IBBE$, corresponds to the model $O3$ in \cite{sdac_roadmap_study} and is the method currently considered by practitioners, see also \cite{PapavasiliouAvila2024} for an extension considering both upward and downward reserves (that can also be obtained as a special case of Theorem \ref{theorem-key-inner-approximation-theorem} after adequate adaptations to handle both upward and downward balancing capacity procurement, see the Supplementary Material). 

It directly relates to three different applications in European and US markets. Firstly, the model is used by practitioners to enforce ``intuitiveness" in European flow-based market coupling. See, for instance, page 47 of \cite{cwe2013}. The approach is also used for extracting ATC domains that can be inscribed within flow-based domains. See, for instance, page 11 of \cite{cwe2020atc}. Finally, the approach can be used in order to enforce the satisfaction of the so-called Simultaneous Feasibility Test in auctions for FTR \emph{options}. See, for instance, the discussion of Figure 3 in \cite{oren2012point}. See also the constraints of long-term flow-based auctions of long-term transmission rights in page 26, Article 41 of \cite{acer2023lttr}.

The model is obtained in two steps. First, auxiliary bilateral exchange variables are introduced. These auxiliary variables enable us to decompose reserve exchanges across zones into a sum of bilateral exchanges:
\begin{equation}
    nr_n = \sum_{z \neq n} (f_{n,z} -  f_{z,n}), \quad n \in N 
\end{equation}

These variables are then substituted for the variables $nr$ in the original flow-based (or DC power flow) constraints \eqref{eq:Base9}, which are reproduced here for convenience, considering the redefinition of the right-hand sides as $\widetilde{F_{k}^{\max}}$ introduced in Section \ref{subsection:base-model}:
%\begin{equation}\label{flowsubstitution}
$\sum_{z \in N}^{}PTDF_{k,z}  nr_z \leq \ \widetilde{F_{k}^{\max}}, \forall k \in CNE$.
%end{equation}

This yields:
\begin{equation}\label{zone-to-zone}
    \sum_{n,z \in N}^{}\left( PTDF_{kn} - PTDF_{kz} \right) {\ f}_{nz} \leq \ \widetilde{F_{k}^{\max}}\ ,\quad k \in CNE 
\end{equation}

In a second step, we invoke Lemma 1 of \cite{bemporad2004inner}. This allows us to describe all boxes fitting within the polyhedron described by equation \eqref{zone-to-zone} which also contain the origin of bilateral exchanges, $f_{n,z}=0 \  \forall n,z \in N$. This results in the constraints \eqref{eq:XIBBE3} below, where $(PTDF_{kn} - PTDF_{kz})$ in equation \eqref{zone-to-zone} is replaced by $\max(PTDF_{kn}$ $- PTDF_{kz}, 0)$.

This leads us to the following formulation, denoted by $IBBE$:
\begin{align}
&  \eqref{eq:SP1}-\eqref{eq:SP5}  \taglabel{IBBE1-IBBE5}\\   
& nr \in Proj_{nr}(XIBBE)                         \taglabel{IBBE6}
\end{align}

The set $XIBBE$ is defined by the following constraints. 
\begin{align}
& \sum_n nr_{n} = 0  \taglabel{XIBBE1} \\
& nr_n = \sum_{z \neq n} (f_{n,z} -  f_{z,n}), \quad n \in N  \quad \quad [\pi rnet_n] \taglabel{XIBBE2} \\
& \sum_{n,z \in N}^{} \max\left( PTDF_{kn} - PTDF_{kz}, 0 \right){\ f}_{nz} \leq \ \widetilde{F_{k}^{\max}}\ ,\quad k \in CNE \quad \quad [SP_k] \taglabel{XIBBE3} \\
& f_{n,z} \geq 0 \taglabel{XIBBE4}
\end{align}

Note that \eqref{eq:XIBBE1} is actually implied by \eqref{eq:XIBBE2}. Nevertheless, we keep it in the description of $XIBBE$ in order to keep consistency with the other models. The model $IBBE$ enforces the reserve deliverability requirement, albeit in a too conservative way. We show this later in this section.

\begin{proposition} \label{proposition:ibbe-valid-inner-approx}
    The model $IBBE$ is an inner approximation of $SP$ in the precise sense that $Proj_{nr}(XIBBE) \subseteq Proj_{nr}(XSP)$. This means that the set of permissible cross-zonal exchanges of balancing capacity in $IBBE$ is all contained in the set of secure exchanges that are modeled in $SP$.
\end{proposition}

\begin{proof}

We show that the model is just a special case of the general construction discussed in Theorem \ref{theorem-key-inner-approximation-theorem}, where the cross-zonal exchanges being combined here are of the form 
$((-1)e_z + 1e_n) f_{n,z}$, where $e_j$ is the unit vector with all entries except $j$ equal to 0 and entry $j$ equal to 1.

Indeed, constraints \eqref{eq:XIBBE2} written with vector notation read:
 \begin{equation}
     \begin{pmatrix}
 nr_{1} \\
 \ldots \\
 nr_{z} \\
 \ldots \\
 nr_{|N|}
 \end{pmatrix} = \begin{pmatrix}
  - 1 \\
 1 \\
 0 \\
 \ldots \\
 0
 \end{pmatrix}f_{2,1} + \ \begin{pmatrix}
  - 1 \\
 0 \\
 1 \\
 \ldots \\
 0
 \end{pmatrix}f_{3,1} + \ldots + \begin{pmatrix}
 0 \\
 1 \\
 0 \\
 \ldots \\
  - 1
 \end{pmatrix}f_{2,|N|} + \ \begin{pmatrix}
 0 \\
 0 \\
 1 \\
 \ldots \\
  - 1
 \end{pmatrix}f_{3,|N|} + \ldots + \begin{pmatrix}
 0 \\
 0 \\
 0 \\
 1 \\
  - 1
 \end{pmatrix}f_{|N| - 1,|N|} \taglabel{XIBBE2-V}
 \end{equation}

This clearly reveals that the vector of net injections, $nr$, can be decomposed into a sum of cross-zonal exchanges of the form \eqref{eq:INAP-1} in Theorem \ref{theorem-key-inner-approximation-theorem}. 

It remains to show that for each of these exchange patterns $e_n - e_z$, $W^{n,z}_k :=\max\left( PTDF_{k,n} - PTDF_{k,z}, 0 \right)$, appearing in \eqref{eq:XIBBE3} -  corresponding to $W^i_k$ in \eqref{eq:INAP-2} in Theorem \ref{theorem-key-inner-approximation-theorem} - are secure capacities in the sense of Definition \ref{def:secure-czc}.

For this exchange pattern $e_n - e_z$, in the spirit of the Netting Lemma \ref{lemma-netting_total_demand_supply}, we define the ``total reserve demand" $TRD$ per zone according to the import position of each zone, namely, $TRD_z := 1$ for zone $z$ and $TRD_l=0$ for all other zones $l \neq z$. Similarly, we define the ``total reserve supply" $TRS$ per zone according to the export position of each zone: $TRS_n := 1$ for zone $n$ and $TRS_l=0$ for all other zones $l \neq n$.

We then consider activation scenarios to demonstrate that the capacities $W^{n,z}_k$ defined above make this matching and cross-zonal exchanges ``reserve deliverable", i.e., $e_n-e_z \in Proj_{nr}(XSP[F_{k}^{\max} \rightarrow W^{n,z}_k])$.

There are two activation scenarios $s_z = 0$ or $s_z=1$ of interest: either $TRDAct_z = 0$, in which case $TRSAct_n=0$ (there is no exchange in that scenario), or $TRDAct_z = 1$, in which case $TRSAct_n=1$ (the initial exchange pattern $e_n - e_z$ in the day-ahead materializes in real time). 

In the first scenario without exchange in real time, on each critical network element $k \in CNE$, the necessary and sufficient capacity to ensure feasibility of the transmission is 0.

In the second scenario where zone $n$ exports $1 MW$ to zone $z$, on each critical network element $k \in CNE$, the necessary and sufficient capacity to ensure feasibility of the transmission is $PTDF_{kn}$ $- PTDF_{kz}$.

Therefore, $max(0, PTDF_{kn}$ $- PTDF_{kz})$ are secure capacities ensuring reserve deliverability in all scenarios for the initial day-ahead exchange pattern $e_n - e_z$.

The proof of the proposition thus follows as a direct corollary of Theorem \ref{theorem-key-inner-approximation-theorem}. 
\end{proof}

The model $IBBE$ scales significantly better than the $SP$ formulation.  Specifically, it requires only a small number of auxiliary variables $f$ that represent bilateral exchanges. These are in the order of $|N|^2$, i.e. polynomial in the number of zones. By contrast, the $SP$ formulation requires an exponential number of auxiliary variables to capture the feasibility of all second-stage scenarios.

However, the example in Figure \ref{fig:feasible_matching} shows that this approach can be conservative in certain instances. The relaxation introduced in the next section aims at relieving this issue. Figure \ref{fig:feasible_matching} also shows that such an improved extension can substantially lower balancing capacity procurement costs. We proceed to explain the example of Figure \ref{fig:feasible_matching}, before advancing to our proposal for overcoming this conservatism.

The example presented in Figure \ref{fig:feasible_matching} assumes a single grid constraint, limiting to  $100 MW$ the capacity of the line $B-C$ in the direction $B$ to $C$, imposing the constraint $\frac{1}{3}NetPosition_A + \frac{2}{3}NetPosition_B \leq 100$.

In this example, the optimal matching of reserve bids that satisfy the reserve deliverability requirement consists of matching 800 MW of demand in $A$ with 300 MW in $C$ and 500 MW in $B$. Indeed, this matching satisfies the network constraints. Moreover, for any percentage of activation of the 800 MW of demand in $A$, we can simply scale down the activations of the supply in $B$ and $C$ by the same percentage, while still satisfying network constraints. However, this matching is not feasible for $IBBE$. Indeed, after transforming the flow-based constraint as described above, one obtains the following constraint on bilateral exchanges: $1/3 f_{B,A} + 2/3 f_{B,C} + 1/3 f_{A,C} \leq 100$. These constraints limit the exports from $B$ to $A$ to 300 MW. Thanks to the less conservative approach which is discussed in the following sections, the entire demand of $800$ MW in $A$ can be served by supply from $B$ (500 MW) and $C$ (300 MW). This reduces reserve procurement costs significantly, since the price in $A$ will decrease by avoiding the matching of the expensive supply at 2000 \euro{}/MWh.

\begin{figure}%[H]
    \centering
    \includegraphics[width=0.6\linewidth]{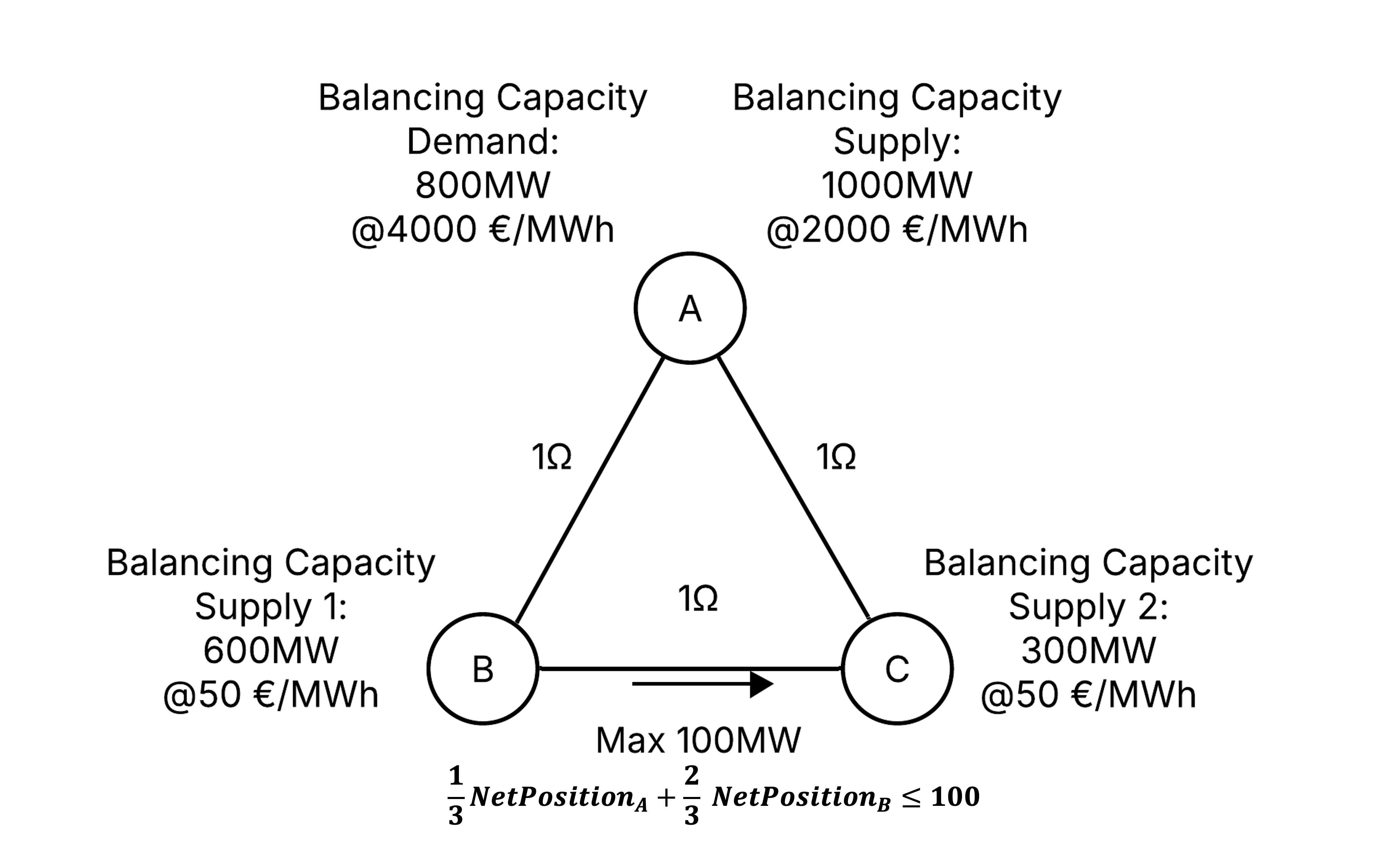}
    \caption{
       An example which illustrates why $IBBE$ can turn out to be a conservative formulation of the reserve deliverability requirement.}
    \label{fig:feasible_matching}
\end{figure}

\section{Inner approximations via boxes in the space of multilateral exchanges}
\label{section:multilateral-exchanges-formulation}

We now elaborate on a relaxation of $IBBE$, which we refer to as $IBME$. The initials stand for inscribed boxes in the space of multilateral exchanges.

The example in Figure \ref{fig:feasible_matching} illustrates a key reason why $IBBE$ may be overly conservative in certain cases. Concretely, in cases where activations in real time are lower than the total procured volume of reserve demand, the model checks for the feasibility of all bilateral exchanges being scaled down \emph{independently} from each other.

However, as illustrated in Figure \ref{fig:feasible_matching}, this can be overly restrictive. In particular, there can be situations where a low activation in some importing location can be managed by uniformly scaling down \emph{simultaneously} multiple bilateral exchanges while still satisfying the network constraints.

This suggests to further leverage the principle outlined in Theorem \ref{theorem-key-inner-approximation-theorem}, by considering in \eqref{eq:INAP-1} the combination of more general exchanges of reserves than the ones used to derive $IBBE$.

Driven by the intuition of the example in Figure \ref{fig:feasible_matching}, we will consider more general exchange patterns than those of \eqref{eq:XIBBE2-V} in $IBBE$, which were of the form $v = (v_{1}^{l}, ..., v_{z}^{l}, ..., v_{|N|}^{l})^T$ with exactly one entry equal to $-1$ and one entry equal to $1$.  Instead, in the more general exchange patterns that we now consider, exactly one zone is importing, while all other zones are only allowed to export. In other words, exactly one entry of $v$ is $-1$ and all other entries in $v$ must be non-negative.

This leads to the model $IBME$ below, and its finite-dimensional reformulation $FIBME$ in Section \ref{section:finite-reformulation}. We will show that this model correctly enforces the reserve deliverability requirement as a direct corollary of Theorem \ref{theorem-key-inner-approximation-theorem}. The advantage of this model is that it is less conservative than $IBBE$. In other words, it is a better inner approximation of $SP$.

The main challenge with situations such as the one that arises in Figure \ref{fig:feasible_matching} is that the repartition of the exports across exporting zones is not known in advance. Concretely, as far as this example is concerned, we do not know in advance how much power will be sourced from $B$ and how much will be sourced from $C$ for any given import in zone $A$.

We therefore need to consider all possibilities. This leads to considering an infinite number of vectors $v$ of cross-zonal exchanges as candidate vectors for the construction in \eqref{eq:INAP-1}-\eqref{eq:INAP-3} of Theorem \ref{theorem-key-inner-approximation-theorem}, even though only a finite number of them will eventually be used. This is where elements of semi-infinite linear programming come into play. We consider in \eqref{eq:XIBME1}-\eqref{eq:XIBME4} an infinite number of decision variables $f_v$. Each such variable is indexed by $v$, which corresponds to each  possible vector of coefficients  of the type described above, i.e. where one entry is equal to -1 and all others are non-negative. Only a finite number of these variables $f_v$ are allowed to take non-zero values. This is called a ``generalized finite sequence'' in the classical work \cite{charnes1963duality} on duality in semi-infinite linear programming. This ensures that the sums appearing in \eqref{eq:XIBME1} are indeed finite, exactly in line with \eqref{eq:INAP-1} in Theorem \ref{theorem-key-inner-approximation-theorem} where a finite number of cross-zonal exchanges are summed up.

\begin{align}
& \max_{pr,dr, nr} \ \ \sum_r VR^r dr_r  \taglabel{IBME1} \\
& dr_r \leq DR_r                & \forall r \in RDB    \hspace{1cm} [sdr_r]                                \taglabel{IBME2}\\
& pr_{g} \leq \widetilde{P_g}          & \forall g \in G       \hspace{1cm} [s^+_g]                              \taglabel{IBME3}\\
&  \sum_{r | n(r) = n} dr_r - \sum_{g | n(g) = n} pr_g = -nr_n  & \forall n \in N  \hspace{1cm} [\pi r_n]       \taglabel{IBME4}\\
& & \notag \\
& dr, pr \geq 0              \taglabel{IBME5}\\
& nr \in Proj_{nr}(XIBME)                         \taglabel{IBME6}
\end{align}

with $XIBME$ defined by the following constraints :

\begin{align}
& \begin{pmatrix}
nr_{1} \\
\vdots \\
nr_{z} \\
\vdots \\
nr_{|N|}
\end{pmatrix}
= 
\sum_{\substack{v_{}^{1}\ | \\ v_{1}^{1} = -1, \\ v_{z}^{1} \geq 0 \ \forall z \neq 1 \\ \sum_z v^1_{z} = 0} }
\begin{pmatrix}
v_{1}^{1} \\
\vdots \\
v_{z}^{1} \\
\vdots \\
v_{|N|}^{1}
\end{pmatrix}
f_{v^1}
+ \ldots 
+ \sum_{\substack{v_{}^{l}\ | \\ v_{l}^{l} = -1, \\ v_{z}^{l} \geq 0 \ \forall z \neq l \\ \sum_z v^l_{z} = 0} }
\begin{pmatrix}
v_{1}^{l} \\
\vdots \\
v_{z}^{l} \\
\vdots \\
v_{|N|}^{l}
\end{pmatrix}
f_{v^l}
+ \ldots +
\sum_{\substack{v_{}^{|N|}\ | \\ v_{|N|}^{|N|} = -1, \\ v_{z}^{|N|} \geq 0 \ \forall z \neq |N| \\ \sum_z v^{|N|}_{z} = 0}}
\begin{pmatrix}
v_{1}^{|N|} \\
\vdots \\
v_{z}^{|N|} \\
\vdots \\
v_{|N|}^{|N|}
\end{pmatrix}
f_{v^{|N|}}  \notag \\ & \hspace{13.4cm}[\pi rnet_n] \taglabel{XIBME1} \\
& \sum_{\substack{v_{}^{1}\ | \\ v_{1}^{1} = -1, \\ v_{z}^{1} \geq 0 \ \forall z \neq 1 \\ \sum_z v^1_{z} = 0} } \max \left(0, \sum_{z}^{}{PTDF_{k,z}}\ v_{z}^{1} \right){\ f}_{v^1} + ... + \sum_{\substack{v_{}^{l}\ | \\ v_{l}^{l} = -1, \\ v_{z}^{l} \geq 0 \ \forall z \neq l \\ \sum_z v^l_{z} = 0} } \max \left(0, \sum_{z}^{}{PTDF_{k,z}}\ v_{z}^{l} \right){\ f}_{v^l}   \notag \\ & + ... + \sum_{\substack{v_{}^{|N|}\ | \\ v_{|N|}^{|N|} = -1, \\ v_{z}^{|N|} \geq 0 \ \forall z \neq |N| \\ \sum_z v^{|N|}_{z} = 0}} \max\left(0, \sum_{z}^{}{PTDF_{k,z}}\ v_{z}^{{|N|}} \right){\ f}_{v^{|N|}} \notag \leq \ \widetilde{F_{k}^{\max}}        ,\quad k \in CNE \quad \quad [SP_k] \taglabel{XIBME2}
\end{align}

\begin{align}
& f_v \geq 0 \taglabel{XIBME3} \\
& f_v \text{ is a generalized finite sequence \cite{charnes1963duality}, i.e. it has finitely many non-zero elements}\taglabel{XIBME4}
\end{align}

\begin{proposition} \label{proposition:ibme_relaxation_of_ibbe}
    The model $IBME$ is a relaxation of $IBBE$
\end{proposition}

\begin{proof}
$IBME$ is a relaxation of $IBBE$ because any feasible point of $IBBE$ can be extended to a feasible point of $IBME$ by setting $f_{v^l}=0$ when $v^l \notin \{v^l | v^l_{l} =-1,  v^l_{z} = 1 \text{\ for a single\ } z \neq l \}$.    
\end{proof}

\begin{proposition}
    The model $IBME$ provides an inner approximation of $SP$ in the precise sense that $Proj_{nr}(XIBME) \subseteq Proj_{nr}(XSP)$. This means that feasible cross-zonal exchanges of balancing capacity in $IBME$ are all contained in the set of secure exchanges modeled in $SP$.
\end{proposition}

\begin{proof}
    As discussed above, this is a direct corollary of Theorem \ref{theorem-key-inner-approximation-theorem}. Given the fact that the sums in \eqref{eq:XIBME1} are finite thanks to the generalized finite sequence condition \eqref{eq:XIBME4}, the constraints \eqref{eq:XIBME1}-\eqref{eq:XIBME4} are a special case of \eqref{eq:INAP-1}-\eqref{eq:INAP-3}, provided that the coefficients $W^l_k :=\max\left(0, \sum_{z}^{}{PTDF_{k,z}}\ v_{z}^{l} \right)$ appearing in \eqref{eq:XIBME2} are secure capacities $W^l_k$ in the sense of Definition \ref{def:secure-czc}.

    The proof that these coefficients $W^l_k$ are secure capacities is then very similar to what is developed in the proof of Proposition \ref{proposition:ibbe-valid-inner-approx}.

    Consider an exchange pattern $v^l = (v^l_1, ..., v^l_l, ..., v^l_{|N|})^T$, where $v^l_l = -1$ and $v^l_z \geq 0$ for $z \neq l$. To show that the $W^l_k$ are secure cross-zonal capacities, we need to build a point of $XSP[F_{k}^{\max} \rightarrow W^l_k]$. To show this, we first define first-stage variable values $TRD, TRS$ in line with the Netting Lemma \ref{lemma-netting_total_demand_supply} as follows: $TRD_l = - v^l_l = 1$ for the importing zone $l$, $TRD_z = 0$ for all other zones $z\neq l$, while $TRS_l = 0$ for $l$ and $TRS_z = v^l_z \geq 0$ for all other zones $z\neq l$.

    As in Proposition \ref{proposition:ibbe-valid-inner-approx}, there are again only two activation scenarios to consider, corresponding to whether $TRD_l$ is fully activated in real time, or if it is not activated at all, given that  $TRD_z = 0$ for all other zones $z \neq l$.

    If there is no activation in real time, $TRDAct_l = 0$, and given the balance conditions \eqref{eq:XSP7}, the activations on the supply side must be all zero as well: $TRSAct_z = 0, z \neq l$. Therefore, in that scenario where there is eventually no exchange in real time, on each critical network element $k \in CNE$, the necessary and sufficient capacity for ensuring feasibility of transmission is 0.

    In the second scenario where all the first-stage reserve demand $TRD_l = -v^l_l = 1$ is activated in real time, one has $TRDAct_l = - v^l_l = 1$, in which case all the corresponding supply $TRS_z = v^l_z, z \neq l$ from the day-ahead needs to be activated, i.e. $TRSAct_z = v^l_z, z \neq l$. In that scenario, on each critical network element $k \in CNE$, the necessary and sufficient capacity for ensuring feasibility of transmission is $\sum_{z}^{}{PTDF_{k,z}}\ v_{z}^{l}$.

    Therefore, $max(0, \sum_{z}^{}{PTDF_{k,z}}\ v_{z}^{l})$ are secure capacities ensuring reserve deliverability in all scenarios for the day-ahead exchange pattern $v^l = (v^l_1, ..., v^l_l, ..., v^l_{|N|})^T$, concluding the proof.

\end{proof}

The challenge with $IBME$ is that it is a semi-infinite dimensional linear program, defined over an infinite number of variables $f_{v^l}$. We address this challenge in two alternative ways. First, we provide in Section \ref{section:finite-reformulation} a finite-dimensional reformulation of $IBME$. For this purpose, we use the (Haar) dual of the semi-infinite linear programming problem $IBME$ and a standard reformulation technique for robust linear optimization problems. Second, in Section \ref{section:cg-algo} we propose a column generation algorithm to tackle $IBME$ and prove its convergence in a finite number of iterations. 

\section{Finite reformulation of %the (infinite-dimensional) formulation for 
boxes in the space of multilateral exchanges} \label{section:finite-reformulation}

In this section, our objective is to rewrite the set of feasible first-stage decisions $Proj_{nr}(XIBME)$, using a \emph{finite} number of variables and linear inequalities, which means that we wish to transform the semi-infinite linear programming formulation \eqref{eq:XIBME1}-\eqref{eq:XIBME4} into a classic linear program. We refer to this as $FIBME$.

We obtain the reformulation $FIBME$ by considering the dual $DIBME$ of the semi-infinite linear programming formulation $IBME$, and then reformulating $DIBME$ into a finite linear program $D$ whose dual is the finite linear program $FIBME$. However, there could be in theory a duality gap between the semi-infinite linear program $IBME$ and its (semi-infinite linear) dual $DIBME$, while there will be by construction no gap between $DIBME$, $D$ and $FIBME$. Instead of relying on standard sufficient conditions ensuring the absence of a duality gap for a semi-infinite linear program and its (Haar) dual, we provide a direct elementary proof that there is no gap between $IBME$ and $FIBME$, therefore showing as a corollary that there is no duality gap between $IBME$ and $DIBME$.

Let us first derive $FIBME$ as outlined above. 

Let us fix an importing zone $l$, appearing as a superscript $l$ in the developments below. In the dual $DIBME$ of $IBME$, the dual constraints associated to the variables $f_{v^l}$ are:

\begin{align}
& \sum_z -{\pi rnet}_z v^l_{z} + \sum_k SP_k \cdot \max(\sum_{z}^{} PTDF_{k,z}v_{z}^{l},0) \geq 0 \label{robust_linear_constr} \\ 
&\forall v^l_{}\ s.t. \notag\\
& v^l_{l}=-1, & [y^l_l] \label{robust_linear_constr-1}\\
& v^l_{z} \geq 0 \hspace{1cm}  l\neq z & [y^l_z] \\
& \sum_z v^l_{z} = 0 & [y^l_0] \label{robust_linear_constr-end}
\end{align}

Note that, in the language of robust optimization, \eqref{robust_linear_constr-1}-\eqref{robust_linear_constr-end} describe the uncertainty set of the parameters $v^l$ of the robust linear constraint \eqref{robust_linear_constr} in the variables $\pi rnet_z$ (dual to constraints \eqref{eq:XIBME1}), $SP_k$ (standing for Shadow Price, dual to constraints \eqref{eq:XIBME2}).

Leveraging a standard technique to formulate a robust linear optimization constraint as a finite set of linear inequalities \cite{sun2021robust}, we show the following.

\medskip

\begin{lemma}\label{lemma:reformulation-dual}
   The robust linear constraint \eqref{robust_linear_constr}, which is dual to $f_{v^l}$ with $v^l$ ranging over all admissible vectors described by \eqref{robust_linear_constr-1}-\eqref{robust_linear_constr-end} can be reformulated as a finite set of linear inequalities in the variables $\pi rnet$ (dual to constraints \eqref{eq:XIBME1}), $SP$ (dual to constraints \eqref{eq:XIBME2}), $h$ (auxiliary variables), $y_0$ (auxiliary variables), as follows:
   \begin{align}
       & - h^l_k + SP_k \geq 0  \label{robust-reformulation-1}\\
       & -y^l_0 + \sum_k PTDF_{k,l} h^l_k \leq {\pi rnet}_l \label{robust-reformulation-2} \\
       & -y^l_0 + \sum_k PTDF_{k,z} h^l_k \geq {\pi rnet}_z \hspace{1cm} z \neq l  \label{robust-reformulation-3} \\
       & h^l_k \geq 0 \label{robust-reformulation-4}
   \end{align}
   
\end{lemma}

\begin{proof}
    Given that the dual variables $SP_k$ corresponding to constraints \eqref{eq:XIBME2} are all non-negative, constraint \eqref{robust_linear_constr} is satisfied if and only if $z\geq 0 $ for
\begin{align}
& z = \min_{v^l_{z}, t_k} \sum_z - \pi rnet_z v^l_{z} + \sum_k SP_k t_k \label{pricing-1} \\ 
& t_k \geq \sum_{z}^{} PTDF_{k,z}v_{z}^{l} & [h^l_k] \\ 
& v^l_{l}=-1,  & [y^l_l] \\
& \sum_z v^l_{z} = 0 & [y^l_0] \\
& v^l_{z} \geq 0 \ \hspace{2cm} \forall z\neq l \\
& t_k \geq 0 \label{pricing-end}
\end{align}

By linear programming duality, $z$ is also equal to: 
\begin{align}
& z = \max_{h,y} - y^l_l \\ 
&  h^l_k  \leq SP_k  \\
& y^l_0 + y^l_l - \sum_k PTDF_{k,l} h^l_k = {-\pi rnet}_l  \\
& y^l_0 - \sum_k PTDF_{k,z} h^l_k \leq {-\pi rnet}_z \hspace{1cm} z \neq l \\
& h^l_k \geq 0
\end{align}

and $z\geq 0$ if and only if the following system is feasible:
\begin{align}
& - y^l_l \geq 0 \label{ref-slack-var-0}\\ 
& h^l_k - SP_k \leq 0 \\
& y^l_0 + y^l_l - \sum_k PTDF_{k,l} h^l_k = {-\pi rnet}_l  \label{ref-slack-var} \\
& y^l_0 - \sum_k PTDF_{k,z} h^l_k \leq {-\pi rnet}_z \hspace{1cm} z \neq l \\
& h^l_k \geq 0
\end{align}

We can then project out $y^l_l$. Given \eqref{ref-slack-var-0}, this acts as a slack variable in \eqref{ref-slack-var}. We specifically remove \eqref{ref-slack-var-0} and replace \eqref{ref-slack-var} by \eqref{ref-slack-var-new} below:
\begin{align}
& h^l_k - SP_k \leq 0 \\
& y^l_0 - \sum_k PTDF_{k,l} h^l_k \geq {-\pi rnet}_l  \label{ref-slack-var-new}  \\
& y^l_0 - \sum_k PTDF_{k,z} h^l_k \leq {-\pi rnet}_z \hspace{1cm} z \neq l \\
& h^l_k \geq 0
\end{align}

The reformulation \eqref{robust-reformulation-1}-\eqref{robust-reformulation-4} then follows by multiplying the first three inequalities by (-1). \end{proof}

Thanks to Lemma \ref{lemma:reformulation-dual}, the dual of $IBME$ can be reformulated as the following finite-dimensional LP. Here, the infinite set of dual constraints \eqref{robust_linear_constr} associated to the primal variables $f_v$ has been replaced by \eqref{eq:D5}-\eqref{eq:D8}.
\begin{align}
& \min \sum_r DR_r sdr_r + \sum_g \widetilde{P_g} s_{g}^{+} + \sum_k \widetilde{F_k^{max}}  SP_k \taglabel{D1}\\
& s_{g}^{+} - \pi r_{z(g)} \geq 0\ \hspace{4cm} g \in G & [pr_g] \taglabel{D2} \\
& sdr_{r}\  + \ {\pi r}_{z(r)} \geq VR^{r}, \hspace{3cm} r \in RDB &[dr_r] \taglabel{D3} \\
& {\pi r}_{z} + \ \pi{rnet}_{z}= 0 & [nr_z] \taglabel{D4} \\
& - h^l_k + SP_k \geq 0 & [w^l_k \geq 0] \taglabel{D5}  \\
& -y^l_0 + \sum_k PTDF_{k,l} h^l_k \leq {\pi rnet}_l \hspace{2cm} \forall l \in N  & [nr^l_l \leq 0]  \taglabel{D6} \\
& -y^l_0 + \sum_k PTDF_{k,z} h^l_k \geq {\pi rnet}_z \hspace{1cm} z \neq l, l\in N  & [nr^l_z \geq 0]  \taglabel{D7}  \\
& h \geq 0 \taglabel{D8} \\ 
&  s_{g}^{+}, sdr, SP \geq 0 \taglabel{D9}
\end{align}

The dual of the reformulated dual \eqref{eq:D1}-\eqref{eq:D9} yields a finite-dimensional reformulation $FIBME$ of $IBME$ given below.  In Proposition \ref{proposition:ibme-fibme}, we provide a direct proof that the two models are equivalent.

Technically, this enables us to show that there is no duality gap between $IBME$ and its semi-infinite linear programming ``Haar dual" $DIBME$, where the infinite set of dual constraints attached to $f_v$ are considered, before reformulation into the finite dual $D$ and then $FIBME$ (these latter reformulations by design preventing  any additional gap).

Showing the absence of a duality gap between $IBME$ and $DIBME$ would otherwise require relying on specific sufficient conditions from the semi-infinite linear programming literature beyond the scope of this paper, with the core reason being that any constraint $c$ implied by the infinite family of constraints \eqref{robust_linear_constr} for all $v$ satisfying  \eqref{robust_linear_constr-1}-\eqref{robust_linear_constr-end} is actually implied by a finite number of them, which is called the Farkas-Minkowski property, see \cite{goberna1999duality} (here due to the compact and polyhedral nature of  \eqref{robust_linear_constr-1}-\eqref{robust_linear_constr-end} leveraged to obtain the reformulation $FIBME$).

\begin{align}
& \max_{pr,dr, nr_z} \ \  \sum_r VR^r dr_r  \taglabel{FIBME1} \\
& dr_r \leq DR_r                & \forall r \in RDB    \hspace{1cm} [sdr_r]                                \taglabel{FIBME2}\\
& pr_{g} \leq \widetilde{P_g}          & \forall g \in G       \hspace{1cm} [s^+_g]                              \taglabel{FIBME3}\\
&  \sum_{r | z(r) = z} dr_r - \sum_{g | z(g) = z} pr_g = -nr_z  & \forall z \in N  \hspace{1cm} [\pi r_z]       \taglabel{FIBME4}\\
& & \notag \\
& dr, pr \geq 0              \taglabel{FIBME5}\\
& nr \in Proj_{nr}(XFIBME)                         \taglabel{FIBME6}
\end{align}

with $XFIBME$ given by:

\begin{align}
& nr_z = \sum_{l \in N} nr^l_z \hspace{2cm} \forall z \in N & [{\pi rnet}_z] \taglabel{XFIBME1} \\
& -\sum_{z \in N} nr^l_z = 0 & \forall l \in N \taglabel{XFIBME2} \\
& \sum_{z \in N} PTDF_{k,z} nr^l_z \leq w^l_k \hspace{4cm} \forall l \in N, k \in CNE \taglabel{XFIBME3} \\
& \sum_{l \in N} w^l_k \leq \widetilde{F_{k}^{\max}} \hspace{4cm} \forall k \in CNE  \taglabel{XFIBME4}\\
& nr^l_l \leq 0, \ \ \ \  nr^l_z \geq 0 \ \ \ z\neq l, \hspace{2cm } \forall l \in N \taglabel{XFIBME5} \\
& w \geq 0 \taglabel{XFIBME6}
\end{align}

\begin{proposition} The models $IBME$ and $FIBME$ are equivalent: \\
    $Proj_{nr}(XIBME) = Proj_{nr}(XFIBME)$
\end{proposition} \label{proposition:ibme-fibme}

\begin{proof}
Let us first show that $Proj_{nr}(XIBME) \subseteq Proj_{nr}(XFIBME)$. Given a point of $XIBME$, we need to construct a point of $XFIBME$ having the same $nr$ values. Consider the right-hand side of \eqref{eq:XIBME1} composed of a sum of series (having each only a finite number of non-zero terms), and for each "destination zone $l$", $l \in \{1, ..., |N|\}$, let us set 

\begin{equation}
nr^l := \sum_{\substack{v_{}^{l}\ | \\ v_{l}^{l} = -1, \\ v_{j}^{l} \geq 0 \ \forall j \neq l \\ \sum_j v^l_{j} = 0} }
\begin{pmatrix}
v_{1}^{l} \\
\vdots \\
v_{z}^{l} \\
\vdots \\
v_{|N|}^{l}
\end{pmatrix}
f_{v^l}  \quad \quad  w^l_k :=  max(0, \sum_z PTDF_{k,z} nr^l_z)
\end{equation}

These values of $nr^l, w^l_k$ exactly give us the needed point in $XFIBME$, since a routine check shows that \eqref{eq:XFIBME1}-\eqref{eq:XFIBME3}, \eqref{eq:XFIBME5}-\eqref{eq:XFIBME6} are satisfied. Moreover, \eqref{eq:XFIBME4} in $XFIBME$ is implied by \eqref{eq:XIBME2} in $XIBME$ since summing up over $l$ the following inequalities shows that for each $k$, $\sum_l w^l_k \leq \widetilde{F_{k}^{\max}}$ because \eqref{eq:XIBME2} requires the right-hand sides summed up over $l$ to be less or equal to $\widetilde{F_{k}^{\max}}$:
\begin{multline}
 w^l_k =  max(0, \sum_z PTDF_{k,z} nr^l_z) = max(0, \sum_{z} \sum_{\substack{v_{}^{l}\ | \\ v_{l}^{l} = -1, \\ v_{j}^{l} \geq 0 \ \forall j \neq l \\ \sum_j v^l_{j} = 0} } PTDF_{k,z}\ v_{z}^{l} f_{v^l} )  \leq \sum_{\substack{v_{}^{l}\ | \\ v_{l}^{l} = -1, \\ v_{j}^{l} \geq 0 \ \forall j \neq l \\ \sum_j v^l_{j} = 0}} max \left(0, \sum_{z}^{}{PTDF_{k,z}}\ v_{z}^{l} \right) f_{v^l}
\end{multline}

Let us now show that $Proj_{nr}(XFIBME) \subseteq  Proj_{nr}(XIBME)$. We construct a point of $Proj_{nr}(XIBME)$ from a point in $Proj_{nr}(XFIBME)$ having the same $nr$ values as follows. Consider destinations $l$ such that $nr^l_l \neq 0$ (otherwise, all entries $nr^l_j = 0$ given \eqref{eq:XFIBME2} and \eqref{eq:XFIBME5} and such null vectors don't need to be considered). We set the following  values for $v^l, f_{v^l}$, where the normalization by $|nr^l_l|$ ensures that $v^l_l = -1$ as in $XIBME$.

\begin{equation}
\begin{pmatrix}
v_{1}^{l} \\
\vdots \\
v_{l}^{l} \\
\vdots \\
v_{|N|}^{l}
\end{pmatrix}
 := \begin{pmatrix}
nr_{1}^{l} \ / \  |nr^l_l| \\
\vdots \\
(nr_{l}^{l} \  /\  |nr^l_l|) = -1 \\
\vdots \\
nr_{|N|}^{l} \  / \  |nr^l_l|
\end{pmatrix}, \quad \quad f_{v^l} := |nr^l_l| \quad \quad l\in \{ 1, ..., |N| \}
\end{equation}

The constraints \eqref{eq:XFIBME3}-\eqref{eq:XFIBME4}, \eqref{eq:XFIBME6} ensure that the constraints \eqref{eq:XIBME2} are satisfied, besides the other constraints  \eqref{eq:XIBME3} which are trivially satisfied since $|nr^l_l| \geq 0$. \end{proof}

The elementary example below demonstrates that the models $IBME, FIBME$ are still conservative. The reason is that not all secure exchanges can be decomposed in the form of a sum of exchanges where only one bidding zone can import while all others can only export, as required by $XFIBME$. The flow-based domain considered in the example of Figure \ref{fig:ibme-fibme-conservative} is described by:
\begin{align}
    & 0.6nr_A + 0.3 nr_B \leq 50 \\
    & -0.6nr_A - 0.3 nr_B \leq 50 \\
    & -\frac{1}{3}nr_A + \frac{1}{3}nr_B - \frac{1}{3}nr_C + \frac{1}{3} nr_D \leq 20 \\
    & \frac{1}{3}nr_A - \frac{1}{3}nr_B + \frac{1}{3}nr_C - \frac{1}{3} nr_D \leq 20 \\
    & nr_A, nr_B, nr_C, nr_D \in [-400; 400]
\end{align}

In Figure \ref{fig:ibme-fibme-conservative}, we show that the feasible sets $Proj_{nr}(XIBME) = Proj_{nr}(XFIBME)$ are strictly contained in the exact stochastic programming formulation for reserve deliverability $Proj_{nr}(XSP)$, and that they strictly contain the feasible set $Proj_{nr}(XIBBE)$.

\begin{figure}%[h!]
    \centering
    \includegraphics[width=1\linewidth]{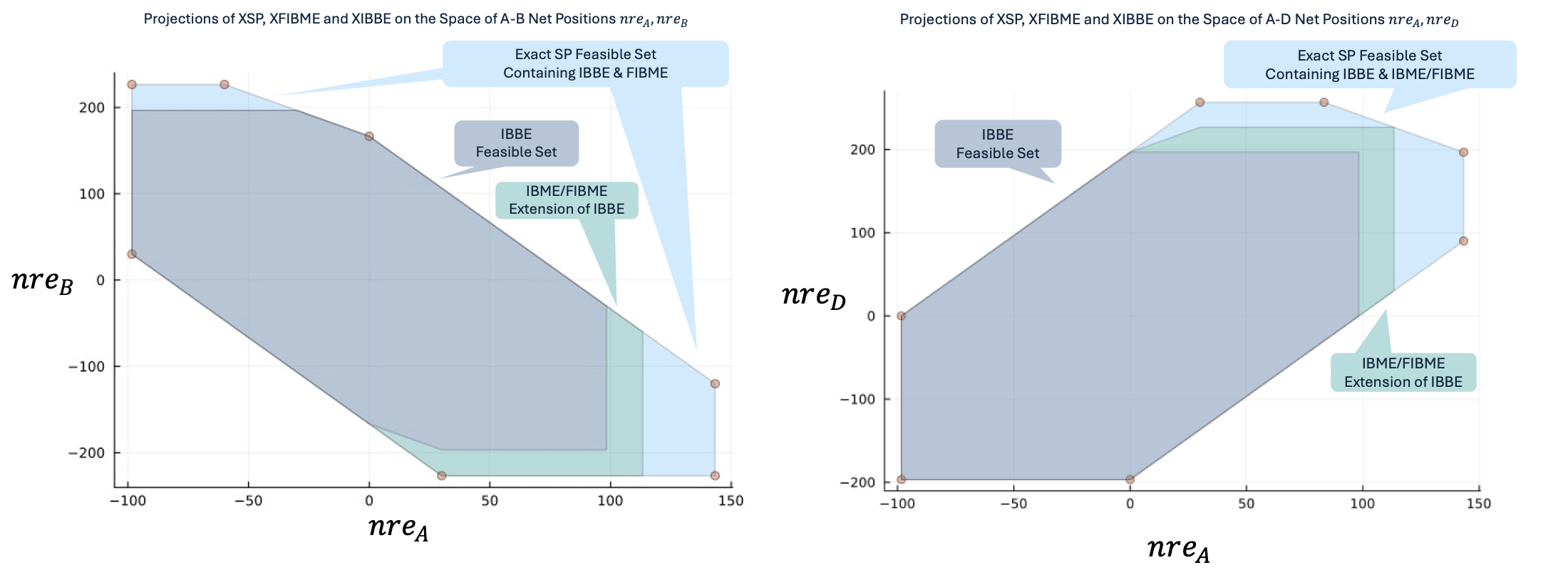}
    \caption{Illustration that IBME/FIBME are inner approximations of the feasible set of reserve deliverability, but less conservative than IBBE.}
    \label{fig:ibme-fibme-conservative}
\end{figure}

The numerical experiments presented in Section \ref{section:numerical-experiments} demonstrate that the formulations $IBME, FIBME$ already allow us to reap substantial benefits of secure cross-zonal exchanges of balancing capacity. 

Nevertheless, it would be interesting to investigate how to further leverage the inner approximation result of Theorem \ref{theorem-key-inner-approximation-theorem} to further alleviate this conservativeness. Various strategies and decomposition methods should allow us to consider the necessary cross-zonal exchanges $nr^i$ to combine in \eqref{eq:INAP-1}, together with adequate (non-unique) secure capacities $W_k^i$, see Theorem \ref{theorem-key-inner-approximation-theorem}. 

Such an analysis is however left for future work, although an exact approach for solving $SP$ that we plan to develop in an extension of this work is evoked at the end of Section \ref{section:cg-algo}.

\section{A column generation algorithm for IBME% and Comparison with a Custom Dantzig-Wolfe Decomposition
}

\label{section:cg-algo}
Given that the formulation $FIBME$ requires $|N|$ copies of the network model per reserve product when there are $|N|$ zones, it may not scale very well in practice. To address this challenge, we revert back to the semi-infinite linear programming formulation $IBME$. We are inspired to develop a decomposition algorithm by the fact that $IBME$ is a relaxation of $IBBE$ where we only consider bilateral exchanges of reserves between zones.

The idea of our proposed algorithm is to start with an optimal solution of $IBBE$ and identify missing columns $f_{v^l}$ in the formulation $IBME$, that would improve welfare if they were added. Such missing columns will correspond to violations of the dual constraints \eqref{robust_linear_constr}. 

The algorithm follows a standard column generation logic. Concretely, consider an optimal solution to $IBBE$ and its associated optimal dual variables, which include $\pi rnet_n$ dual to constraints \eqref{eq:XIBBE2} and $SP_k$ dual to constraints \eqref{eq:XIBBE3}, noting that the constraints \eqref{eq:XIBBE2}  are the same as \eqref{eq:XIBME1} and \eqref{eq:XIBBE3} the same as \eqref{eq:XIBME2} for this \emph{initial Restricted Master} $IBBE$ (see Proposition \ref{proposition:ibme_relaxation_of_ibbe} in Section \ref{section:multilateral-exchanges-formulation} on why $IBBE$ is a restricted version of $IBME$).

The pricing problem of our column generation algorithm identifies a vector of coefficients $v^l$ of a column $f_{v^l}$ corresponding to a violation of a robust dual constraint of the form \eqref{robust_linear_constr}. Such a column is obtained by solving the pricing problem \eqref{pricing-1}-\eqref{pricing-end}.

Note that at most a finite number of columns per robust constraint of type \eqref{robust_linear_constr} can be added, since the feasible set of the pricing problem for each $l\in N$ has a finite number of vertices, under the classic assumption that the pricing problems are solved with e.g., the (dual) simplex algorithm only returning optimal vertices or extreme rays. Note also that, at each iteration, as in a standard Dantzig--Wolfe decomposition, a dual (upper) bound can be obtained from the objective values of the current restricted master problem and the pricing problem; see, e.g., \cite{bazaraa2011linear} or \cite{conforti2014integer}. We however do not discuss this bound further here, since all instances in the numerical experiments presented in Section~\ref{section:numerical-experiments} are solved to optimality.

Algorithm \ref{alg:column_generation} below is a complete pseudocode description of the proposed column generation algorithm.

Numerical experiments with this approach are presented in Section \ref{section:numerical-experiments}.

Note that, although beyond the scope of what can be developed in this work, an exact, more general column generation algorithm also leveraging Theorem \ref{theorem-key-inner-approximation-theorem} can be developed, adequately generating vertices of the formulation $XSP$ while combining them via Theorem \ref{theorem-key-inner-approximation-theorem}. We plan to develop this general column generation algorithm, its proof of convergence and comparison to a Dantzig-Wolfe decomposition in an extension of the present work.

\begin{algorithm}%[H]
\caption{Column Generation Algorithm for Solving IBME}
\label{alg:column_generation}
\begin{algorithmic}[1]
\State Initialize the Restricted Master Problem (RMP) with the initial formulation $IBBE$ which is by design a restriction of $IBME$ with a subset of columns corresponding to bilateral exchanges (see Proposition \ref{proposition:ibme_relaxation_of_ibbe} in Section \ref{section:multilateral-exchanges-formulation}).
\State Initialize the lower bound $LB$ on the optimal value of $IBME$ by setting $LB:=ObjectiveValue(IBBE)$.
\Repeat
    \State Solve the current RMP, obtain the optimal primal and dual solutions, and update $LB$.
    \State Solve the pricing problem \eqref{pricing-1}-\eqref{pricing-end} for each $l \in N$ using the dual variable values $\pi rnet_n$ dual to constraints  \eqref{eq:XIBME1} and $SP_k$ dual to constraints \eqref{eq:XIBME2}.
    \If{at least one of the pricing problems takes a negative optimal objective value}
        \State Add the corresponding column $v^l$ and its associated decision variable $f_{v^l}$ to the RMP.
    \Else
        \State \textbf{Terminate}; the current RMP solution is optimal for the full master problem.
    \EndIf
\Until{no improving column exists}
\State \Return the optimal solution of the full master problem $IBME$.
\end{algorithmic}
\end{algorithm}

\section{Numerical Experiments}
\label{section:numerical-experiments}

The previously discussed models are applied to the CORE region transmission system for a specific hour over a 10-day period. The zones modeled include Austria, Belgium, Croatia, the Czech Republic, France, Germany, Hungary, Luxembourg, the Netherlands, Poland, Romania, Slovakia and Slovenia. Flow-based data are sourced from \cite{Jao}. Reserve requirements, their valuations for each zone and generator installed capacities are derived from the dataset used in \cite{Reserve}. The current experimental setup strictly focuses on the exchange of balancing capacity, excluding energy, and installed generation capacities are scaled down accordingly. The latter adjustment is driven by the fact that since only balancing capacity is considered, it is implicitly assumed that part of the generation capacity is already allocated to the energy market cleared separately. 

The models are implemented in Julia v.1.10.5 using JuMP v1.23.4, on a custom PC with an AMD Ryzen 9 7900X 12-Core Processor (4.70 GHz) on Windows 11 (64-bit). The chosen linear programming solver is Gurobi 11. 

We define a metric, the cross-zonal value capture, to quantify the extent to which each model captures the total value generated by the utilization of available cross-zonal capacity in the exact $SP$ formulation. This metric represents the fraction, captured by each model,
of the maximum possible value that can be generated by cross-zonal exchanges satisfying the reserve deliverability requirement. This maximum value achieved by the exact model $SP$ is obtained by comparing the welfare generated by $SP$ and a counterfactual scenario where there is no cross-zonal capacity (NCBC), i.e., by $Welfare_{SP} - Welfare_{NCBC}$. 

For each model, the metric is therefore defined as the ratio: $ U := \frac{Welfare_{model} - Welfare_{NCBC}}{Welfare_{SP} - Welfare_{NCBC}} \cdot 100\% $. The $SP$ model achieves the best value of 100 \%. The objective is to assess how much of the value generated by secure cross-zonal exchanges can be captured by the other models.

Table \ref{ComparisonB} reports on this metric, together with welfare and runtime results for each model across 10 instances. The runtime denotes the total computational time in seconds required by the solver to reach the optimal solution for each model. 

The results highlight significant differences in computational efficiency, as well as achieved welfare and cross-zonal value capture among the formulations. For instance, the compact formulation $IBBE$ is solved almost instantaneously, with a runtime in the order of milliseconds. However, its computational efficiency comes at a cost in terms of achieved welfare and cross-zonal value capture.

Table \ref{ComparisonB} demonstrates that the $IBME$ model, solved via the column generation algorithm outlined in Section \ref{section:cg-algo}, strikes the best balance between computational performance and cross-zonal value capture. Compared to $IBBE$, it substantially improves the value capture, from 5 to up to nearly 14 percentage points, while still being solved in the order of tens of milliseconds. Table \ref{ComparisonB} also shows that the welfare achieved by $IBME$ also consistently matches the welfare achieved by $FIBME$, providing an empirical confirmation of the exactness of the finite-dimensional reformulation $FIBME$ that is proved in Section \ref{section:finite-reformulation}. Note that the reported computational time for $IBME$ includes the first iteration, in which the initial column is generated from the dual values of $IBBE$ and thus, also accounts for the time to solve $IBBE$ itself for this warm start. As Table \ref{ComparisonB} shows, $FIBME$'s runtime is slightly worse than that of $IBME$, while still being of the same order of magnitude. 

Moreover, in 8 instances out of 10, $IBME$ and $FIBME$ allow to capture more than 97\% of the total value generated by cross-zonal exchanges in $SP$ while being solved in a fraction of the time required to solve the exact model $SP$, which has a runtime several orders of magnitude larger than the runtime needed to solve $IBME$ or $FIBME$.

\begin{table*}[htbp]
\centering
\caption{Comparison of IBME, FIBME, and bid-agnostic SP in extensive form across 10 instances.}
\renewcommand{\arraystretch}{1.6}
\setlength{\tabcolsep}{4pt}
\resizebox{\textwidth}{!}{%
\begin{tabular}{|c|c|ccc|cccc|ccc|ccc|}
\hline
\textbf{Inst.} 
& \multicolumn{1}{c|}{\textbf{NCBC}}
& \multicolumn{3}{c|}{\textbf{IBBE}}
& \multicolumn{4}{c|}{\textbf{IBME}}
& \multicolumn{3}{c|}{\textbf{FIBME}}
& \multicolumn{3}{c|}{\textbf{SP in Extensive Form }}
\\ \cline{2-15}
& \textbf{Welfare}
& \textbf{Welfare} & \textbf{U[\%]} & \textbf{Rt[s]}
& \textbf{Welfare} & \textbf{U[\%]} & \textbf{Rt[s]} & \textbf{Generated Cols}
& \textbf{Welfare} & \textbf{U[\%]} & \textbf{Rt[s]}
& \textbf{Welfare} & \textbf{U[\%]} & \textbf{Rt[s]}
\\ \hline
1 & 3,739,213 & 4,791,044 & 93.22 & 0.011 & 4,852,026 & 98.62 & 0.036 & 12 & 4,852,026 & 98.62 & 0.054 & 4,867,555 & 100.00 & 39.544 \\
2 & 3,739,213 & 4,858,146 & 92.47 & 0.015 & 4,917,515 & 97.37 & 0.046 & 9 & 4,917,515 & 97.37 & 0.046 & 4,949,312 & 100.00 & 114.803 \\
3 & 3,739,213 & 4,831,424 & 87.81 & 0.004 & 4,980,520 & 99.79 & 0.046 & 12 & 4,980,520 & 99.79 & 0.061 & 4,983,106 & 100.00 & 54.435 \\
4 & 3,739,213 & 4,684,407 & 78.27 & 0.005 & 4,751,876 & 83.86 & 0.037 & 8 & 4,751,876 & 83.86 & 0.062 & 4,946,796 & 100.00 & 89.229 \\
5 & 3,739,213 & 4,702,951 & 78.65 & 0.000 & 4,795,013 & 86.16 & 0.033 & 12 & 4,795,013 & 86.16 & 0.064 & 4,964,573 & 100.00 & 56.660 \\
6 & 3,739,213 & 4,814,779 & 90.84 & 0.003 & 4,920,563 & 99.77 & 0.018 & 8 & 4,920,563 & 99.77 & 0.062 & 4,923,278 & 100.00 & 157.510 \\
7 & 3,739,213 & 4,796,652 & 92.78 & 0.000 & 4,863,207 & 98.62 & 0.048 & 9 & 4,863,207 & 98.62 & 0.058 & 4,878,968 & 100.00 & 132.486 \\
8 & 3,739,213 & 4,940,160 & 84.75 & 0.000 & 5,140,811 & 98.91 & 0.018 & 11 & 5,140,811 & 98.91 & 0.051 & 5,156,286 & 100.00 & 42.306 \\
9 & 3,739,213 & 4,862,356 & 91.85 & 0.000 & 4,942,499 & 98.41 & 0.016 & 12 & 4,942,499 & 98.41 & 0.061 & 4,962,001 & 100.00 & 46.543 \\
10 & 3,739,213 & 4,964,283 & 94.08 & 0.000 & 5,033,855 & 99.42 & 0.044 & 8 & 5,033,855 & 99.42 & 0.068 & 5,041,435 & 100.00 & 141.641 \\
\hline
\end{tabular}
}
\label{ComparisonB}
\end{table*}

As mentioned in Section \ref{subsection:SP}, we also investigated decomposition techniques for the  bid-agnostic $SP$ formulation \eqref{eq:SP1}-\eqref{eq:SP6}, \eqref{eq:XSP1}-\eqref{eq:XSP9}, implementing a Dantzig-Wolfe (DW) decomposition (see \cite{conforti2014integer, bazaraa2011linear}).

To enable a meaningful comparison with the $IBME$ model, the termination criterion of the Dantzig-Wolfe decomposition was set to the point at which the lower bound first reaches or surpasses the $IBME$ objective value. The results show that the initial bid-aware $InitSP$ formulation \eqref{eq:InitSP1}-\eqref{eq:InitSP15} performs the worst, as it considers scenario-dependent bid acceptances in the second stage, which results in a greater number of variables and constraints than the bid-agnostic one. %The L-shaped approach still requires runtimes on the order of several hundred seconds, indicating its limited computational efficiency. 

Both the bid-agnostic $SP$ and its DW decomposition have comparable runtimes under the chosen termination criterion for the DW decomposition, but the model in extensive form has already attained the exact optimum, whereas the DW decomposition provides only a lower bound that remains within a small gap of it. 

Driving DW to full convergence would close this gap at the cost of additional runtime, so the extensive formulation is preferable to reach exact solutions to $SP$.

\begin{table*}[htbp]
\centering
\caption{Comparison of bid-agnostic $SP$ in extensive form, bid-aware $InitSP$, and $SP$ Dantzig-Wolfe across 10 instances, with $SP$ Dantzig-Wolfe stopped after reaching a solution as good as the one of $IBME$.}
\renewcommand{\arraystretch}{1.6}
\setlength{\tabcolsep}{4pt}
\resizebox{\textwidth}{!}{%
\begin{tabular}{|c|cc|cc|ccc|}
\hline
& \multicolumn{2}{c|}{\textbf{SP in Extensive Form}} 
& \multicolumn{2}{c|}{\textbf{Bid-aware InitSP}}
& \multicolumn{3}{c|}{\textbf{SP Dantzig-Wolfe}} \\ \hline
\textbf{Inst.} &
\textbf{Welfare [€]} & \textbf{Runtime [sec]} &
\textbf{Welfare [€]} & \textbf{Runtime [sec]} &
\textbf{LB [€]} & \textbf{Runtime [sec]} & \textbf{Columns} \\ \hline

1  & 4,867,555 & 39.544 & 4,867,555 & 932.906 & 4,866,871 & 117.119 & 6 \\

2  & 4,949,312 & 114.803 & 4,949,312 & 934.687 & 4,923,637 & 257.867 & 4 \\

3  & 4,983,106 & 54.435 & 4,983,106 & 974.172 &  4,983,064 & 100.582 & 5 \\

4  & 4,946,796 & 89.229 & 4,946,796 & 971.214 &  4,752,740 & 151.033 & 6 \\

5  & 4,964,573 & 56.660 & 4,964,573 & 986.012 &  4,795,319 & 351.732 & 12 \\

6  & 4,923,278 & 157.510 & 4,923,277 & 984.093 &  4,922,376 & 224.93 & 8 \\

7  & 4,878,968 & 132.486 & 4,878,965 & 956.745 & 4,874,025 & 79.334 & 4 \\

8  & 5,156,286 & 42.306 & 5,156,286 & 976.492 & 5,149,156 & 208.774 & 10 \\

9  & 4,962,001 & 46.543 & 4,962,001 & 928.494 & 4,953,845 & 154.365 & 6 \\

10 & 5,041,435 & 141.641 & 5,041,435 & 918.703 & 5,038,091 & 93.050 & 4 \\ \hline

\end{tabular}
}
\label{SP_combined}
\end{table*}

Finally, Table \ref{numbers} reports on the number of variables and constraints encountered in each model or solution method. It provides for instance indications on why solving the model $FIBME$ takes more time than solving $IBME$ via the column generation algorithm outlined in Section \ref{section:cg-algo}, given the substantially larger number of variables and constraints involved, since $FIBME$ considers $|N|$ copies of the original network model.

\begin{table*}[htbp]
\centering
\caption{Number of constraints and variables across models. Reported values correspond to the instances with the largest number of constraints and variables. For the column generation algorithms, the number of constraints and variables are determined based on the instance with the most columns added.}
\renewcommand{\arraystretch}{1.2}
\setlength{\tabcolsep}{3pt}
\resizebox{\textwidth}{!}{%
\begin{tabular}{|c|c|c|c|c|c|c|}
\hline
\textbf{Models} &  \textbf{IBBE} & \textbf{IBME} & \textbf{FIBME} & \textbf{SP in Extensive Form} & \textbf{Bid-aware InitSP}  & \textbf{SP Dantzig-Wolfe}\\ \hline %& \textbf{SP L-shaped}
\#\textbf{Constraints} & 3,368 & 3,315 & 6,543  & 818,902 &  12,947,879  &  793,722\\ \hline
\#\textbf{Variables} & 1,715 & 1,639 & 3,283 & 176,633 & 6,224,836 &  173,583\\ \hline
\end{tabular}
\label{numbers}
}
\end{table*}

\section{Conclusions}
\label{section:conclusions}

Reserve deliverability requires that balancing capacity be procured in such a way that any activation pattern of the balancing capacity demands can be managed by the network in real time. However, certain cross-zonal exchanges can take place only if other exchanges occur concurrently, while real-time activations of balancing capacity procured in day-ahead markets are inherently uncertain. Specific models and methods are therefore needed to ensure that reserves can always be delivered where they are needed.

While ensuring reserve deliverability in \emph{all} activation scenarios can be modeled as a stochastic programming (SP) problem, the size of the corresponding formulation grows exponentially with the number of locations in the network. This does not scale well in practice.

We have first shown that activation scenarios can be expressed in a way which is agnostic to the order book of the underlying market model. This reduces the challenge to a network modeling problem. It also already reduces substantially the size of the problem, leading to improved performances.

A general inner approximation theorem has then been introduced. It facilitates the derivation of both existing practitioner models as special cases, and novel, improved models that are less conservative while maintaining computational efficiency. Beyond enabling to reach better solutions in terms of welfare, the improved models can potentially substantially reduce balancing capacity procurement costs. This is of major importance for future Pan-European cross-zonal balancing capacity markets or co-optimization in the Single Day-ahead Coupling. These improvements however implicitly require using flow-based at this activation stage, while balancing energy platforms in Europe are currently based on ATC network models. Implementing flow-based at the activation stage, aligning the practice with the day-ahead balancing capacity procurement stage, is certainly an interesting evolution for European power markets in the years to come.

The new models which have been proposed specifically rely on a semi-infinite linear programming formulation tackled by a custom column generation algorithm, and an equivalent finite-dimensional reformulation of that model which leverages standard robust linear optimization techniques.

Numerical experiments have been conducted to compare the scalability and conservativeness of these inner approximations to the exact stochastic programming formulation of the reserve deliverability requirement. This exact formulation has been solved as an extensive form linear program but also via a standard Dantzig-Wolfe decomposition. These numerical experiments demonstrate that the inner approximations are much more scalable than the exact stochastic programming formulation, while reaping most of the value generated by cross-zonal exchanges in the exact formulation. 

Further leveraging the inner approximation Theorem \ref{theorem-key-inner-approximation-theorem} to derive scalable exact algorithms solving the exact formulation is an interesting venue for extensions of this work, as shortly evoked in Section \ref{section: general-inner-approx-theorem} and Section \ref{section:cg-algo}.

These approaches are of particular interest for future pan-European cross-zonal balancing capacity markets, and can also accommodate co-optimization of energy and balancing capacity products. 

Since the approaches are bid agnostic, we consider them to also be relevant in market clearing models beyond the European context, thereby significantly broadening the relevance of the work presented here.

% \newpage

\section*{Declaration of generative AI assistance.}
OpenAI’s Astra model has been used to examine the mathematical arguments in the first complete draft version of the manuscript deposited on arXiv (arXiv:2609.00439v1). This examination helped identify a gap in the original proof of Theorem 1 and formulate a shorter proof based on a direct elementary argument. It also helped identify and correct a minor issue in Case 1(a) of the proof of Lemma 1, using reasoning consistent with that employed in Case 1(b) and Case 2 (restoring a symmetry in the proofs of both cases). In addition, it helped identify that the model of \cite{PapavasiliouAvila2024} had been stated incorrectly in the Supplementary Material, resulting in a formulation incompatible with the equivalence claimed in Proposition 1 therein (now Proposition S1). The proposition remains valid when the original model is stated correctly. Inconsequential typos have also been identified and corrected. Full responsibility is taken for the resulting updates and the contents of this article.

\appendix

\vspace{-0.5cm}

\section{Lemmas and Proofs of Results Omitted in the Main Text} \label{appendix:proofs}

\bigskip

\textbf{Proof of Lemma \ref{lemma-netting_total_demand_supply}}

\vspace{-0.4cm}
\begin{proof}

We treat separately the following two cases in a symmetric way. Case 1 deals with the case where $nr_z \leq 0$, i.e. the zone is importing or has a zero net position, while Case 2 deals with the case where $nr_z >  0$, i.e. the zone is exporting.

 \textbf{Case 1}: $z$ imports reserves or has zero balance at the day-ahead stage, i.e. $nr^*_z = (TRS^*_z -  TRD^*_z)  \leq 0$. The case is proved in two parts (a) and (b).

    \textbf{Part (a).} Suppose that for all scenarios $s = s_1, ..., s_{|N|} \in \{0;1\}^{|N|}$, the second-stage (scenario-specific) net positions are such that $nr^*_z \leq nrAct^*_{z,s} \leq 0$.  Then necessarily, $nrAct^*_{z,s} = s_z nr^*_z$. 
    
    Indeed, (I) for scenarios $s$ with $s_z=1$, necessarily, $nrAct^*_{z,s} \leq nr*_z$. The reason, re-detailed again below in Case 1 (b) and Case 2, is that if $TRDAct_{z,s} = s_z TRD_{z} = TRD_{z}$, since the local supply in the second-stage $TRSAct_{z,s}$ is bounded by the supply matched in the first-stage $TRS_{z}$, the zone must continue to import (at least) the same as in the first stage. Combined with $nr^*_z \leq nrAct^*_{z,s}$, this leads to $nrAct^*_{z,s} = nr^*_z = s_z nr^*_z$ in this case. (II) If instead, $s_z = 0$, then necessarily, $TRDAct_{z,s} = 0$ and therefore, $nrAct_{z,s} = TRSAct_{z,s} - TRDAct_{z,s} \geq 0$, which combined with $nrAct^*_{z,s} \leq 0$ yields $nrAct^*_{z,s} = 0 = s_z nr^*_z$ in that case.

    Therefore, the assumption $nr^*_z \leq nrAct^*_{z,s} \leq 0$ ensures that none of the following new value definitions, as required by the Lemma statement, leads to a violation of the  $XSP$ constraints:
    
    \begin{itemize}
        \item $TRD_z^\# :=  TRD^*_z - TRS^*_z = -min(nr^*_z,0)$, 
        \item $TRDAct^\#_{z,s} := s_z TRD_z^\#$
        \item $TRS^\#_{z} :=0 = max(nr^*_z,0)$, 
        \item $TRSAct^\#_{z,s} :=0$
    \end{itemize}
    
    This ``netting" leaves the net injections unchanged, i.e., $nr^\#_z = nr^*_z, \ \  nrAct^\#_{z,s} = nrAct^*_{z,s}$ . Therefore, all constraints of $XSP$ remain satisfied, including the scenario-dependent flow-based constraints \eqref{eq:XSP8}.

    \textbf{Part (b).} We conclude this Case 1 in which it was assumed that $nr^*_z \leq nrAct^*_{z,s} \leq 0$, by showing that in each scenario, this assumption can indeed always be done.

    \textbf{(I)} Let us first show that one can always assume $nrAct^*_{z,s} \leq 0$. Suppose that for some scenario $\overline{s} = \overline{s_1}, ..., \overline{s_{|N|}} \in \{0;1\}^{|N|}$, we have $nrAct^*_{z,\overline{s}} > 0$, despite having the first-stage net position $nr^*_z \leq 0$.

    This is only possible if $\overline{s_z} = 0$. Indeed, otherwise, as long as $\overline{s_z} = 1$, i.e. the activated reserve remains equal to the procured reserve (i.e. $ TRDAct^*_{z,\overline{s}} = TRD^*_{z}$), zone $z$ must continue to import. This is because the supply  locally available is limited in each scenario by the first-stage supply available locally given the constraint $TRSAct^*_{z,s} \leq TRS^*_{z}$, and because this supply is not enough to cover the local demand, since by assumption of Case 1, $nr^*_z \leq 0$, and therefore $TRS^*_{z} \leq TRD^*_{z}$.

    Let us now show that the  position $nrAct^*_{z,\overline{s}} = 0$, trivially satisfying  $nrAct^*_{z,\overline{s}} \leq 0$ as needed, can be achieved under the same reserve activation scenario $\overline{s}$.

    First, note that from the point of view of the local balance of reserve demand and supply, this net position $nrAct^*_{z,\overline{s}} = 0$ can be achieved by reducing the value of $TRSAct^*_{z,s}$, thereby reducing the exports.
    
    The main remaining challenge is to show that such a net position $nrAct^*_{z,\overline{s}} = 0$ is compatible with the net positions of the other bidding zones that can be reached under the scenario $\overline{s}$ being considered. Note that to maintain global balance, these other net positions may need to be adapted. Moreover, since the scenario completely determines $TRDAct$ via \eqref{eq:XSP4}, the adaptations of these net positions correspond to adaptations of the $TRSAct$.

    Let us consider the ``complementary scenario" $s^C$ defined as $s^C_z =1$ and $s^C_l = \overline{s_l}, l \neq z$, where $\overline{s_z}=0$ is replaced by $s^{C}_z =1$, for which necessarily, $nrAct_{z,s^C} \leq 0$ for the exact same reasons as those highlighted above (limited available scenario-dependent supply $TRSAct^*_{z,s^C}$ constrained by $ TRSAct^*_{z,s^C} \leq TRS^*_{z}$ and $TRS^*_{z} \leq TRD^*_{z}$ by assumption of Case 1 where $nr^*_z \leq 0$).

    Since both  scenarios $\overline{s}, s^{C}$ are feasible under $XSP$, this means that the inequalities \eqref{eq:XSP4N}-\eqref{eq:XSP9N} below are feasible, whether the parameter $nrAct_{z}^{as-param}$ equals $nrAct^*_{z,\overline{s}} > 0$ or  $nrAct^*_{z,s^C} \leq 0$. 
    
    We now use Lemma 1 in \cite{geoffrion} stating the following. Let $g$ be a convex function and let $Y := \{ y \in \mathcal{R}^m: g(x) \leq y \texttt{\ for some\ } x\}$; then $Y$ is a convex set.
    
    The net position $nrAct^{new}_{z,\overline{s}} = 0$ is in the convex hull of these positions $nrAct^*_{z,\overline{s}} > 0$ and  $nrAct^*_{z,s^C} \leq 0$ for which the inequalities are feasible. By Lemma 1 in \cite{geoffrion} cited above, the inequalities therefore remain feasible for $nrAct_{z,s}^{as-param} := 0 = nrAct^{new}_{z,\overline{s}}$: This net position  is therefore also compatible with the net positions in other zones that can be reached under scenario $\overline{s}$, in the sense of leading to a feasible point of $XSP$, by adequately adjusting, if needed, the $TRSAct$ while satisfying \eqref{eq:XSP5N}. This concludes the proof for this case.

\begin{align} 
& TRDAct_{n,\overline{s}} = \overline{s}_n \ TRD_{n} &\forall n \neq z \in N \ (Second\ Stage) \taglabel{XSP4N} \\
& TRSAct_{n,\overline{s}} \leq TRS_{n} &\forall n \neq z \in N \  (Second\ Stage)  \taglabel{XSP5N} \\
&  nrAct_{n,\overline{s}} - TRSAct_{n,\overline{s}} + TRDAct_{n,\overline{s}} = 0 &\forall n \neq z \in N \  (Second\ Stage)  \taglabel{XSP6N} \\
& \sum_{n \neq z} nrAct_{n,\overline{s}} = - nrAct_{z}^{as-param} &\ (Second\ Stage) \taglabel{XSP7N} \\
& \sum_{n \neq z \in N}^{} PTDF_{k,n} nrAct_{n,\overline{s}}  \leq \ \widetilde{F_{k}^{\max}} - PTDF_{k,z} nrAct_{z}^{as-param} & \ (Second\ Stage) \taglabel{XSP8N} \\
& TRSAct, TRDAct \geq 0  &\ (Second\ Stage)\taglabel{XSP9N}
\end{align}

\textbf{(II)} Let us now show that one can always assume $nr^*_z \leq nrAct^*_{z,s}$. If for some scenario $\overline{s}$, $nrAct^*_{z,\overline{s}} < nr^*_z$, necessarily, the scenario is such that $\overline{s}_z = 1$,  since as discussed in Case 1 part (a) above, if $\overline{s}_z=0$, then $nrAct^*_{z,\overline{s}} \geq 0$.

In such a situation, one considers the complementary scenario defined as $s^C_z =0$ and $s^C_l = \overline{s_l}, l \neq z$, where $\overline{s_z}=1$ is replaced by $s^{C}_z =0$, necessarily implying $nrAct^*_{z,s^C} \geq 0$. This implies that \eqref{eq:XSP4N}-\eqref{eq:XSP9N} remains feasible whether $nrAct_{z}^{as-param} < nr^*_z$ or $nrAct_{z}^{as-param} \geq 0$.

As above, Lemma 1 in \cite{geoffrion} then ensures that there exists a solution where $nrAct^*_{z,\overline{s}} \geq nr^*_z$ for the initial scenario $s$ being considered, because \eqref{eq:XSP4N}-\eqref{eq:XSP9N} remains feasible for any $nrAct_{z}^{as-param}$ in the convex hull of $\{nrAct^*_{z,\overline{s}} < nr^*_z, nrAct^*_{z,s^C} \geq 0\}$.

\medskip

Case 2: $nr_z > 0$, i.e. the zone $z$ is exporting.

The treatment is very similar to Case 1 and we proceed again in two steps (a) and (b). 

(a) First, one shows that if $0 \leq nrAct^*_{z,s} \leq nr^*_z$ in each scenario, i.e. that there is no scenario where the export of the bidding zone is increasing in real time by leveraging the lower reserve activation $s_z\  TRDAct_{z,s} < TRD_{z}$, then the required netting is possible. 

%\stackrel{def}{=}
Indeed if $0 \leq nrAct^*_{z,s} \leq nr^*_z = TRS^*_z - TRD^*_z$, the following netting (cf. Lemma statement) is feasible. (1) Firstly, thanks to $nrAct^*_{z,s} \leq nr^*_z = TRS^*_z - TRD^*_z$, the newly defined $TRSAct^\#_{z,s} := TRSAct^*_{z,s} - TRDAct^*_{z,s}$ are ensured to satisfy $TRSAct^\#_{z,s} \leq TRS^\#_{z}$ in all scenarios, even when $TRDAct^\#_{z,s} = 0$, as required by $XSP$ (supply at the activation stage is bounded by supply at the day-ahead procurement stage). (2) Secondly, thanks to $0 \leq nrAct^*_{z,s}$, the newly defined $TRSAct^\#_{z,s}$ are ensured to satisfy $TRSAct^\#_{z,s} \geq 0$.

\begin{itemize}
        \item $TRD_z^\# := 0 $, 
        \item $TRDAct^\#_{z,s} := 0 $
        \item $TRS^\#_{z} := TRS^*_z - TRD^*_z  = max(nr^*_z,0)$, 
        \item $TRSAct^\#_{z,s} := TRSAct^*_{z,s} - TRDAct^*_{z,s}  \geq 0$
    \end{itemize}

(b) Then, using similar arguments as in Case 1, we show that in each scenario, one can indeed assume that $0 \leq nrAct_{z,s} \leq nr_z$, by relying again on  the consideration of "complementary scenarios".

If for some scenario $\overline{s}$, $nrAct^*_{z,\overline{s}} > nr^*_z$, necessarily, the scenario is such that $\overline{s}_z = 0$ since, similarly as discussed in Case 1 above, if $\overline{s}_z=1$, since $TRSAct^*_{z,\overline{s}} \leq TRS^*_{z}$, the bidding zone $z$ couldn't export more in such scenarios where $s_z = 1$ as this implies $TRDAct_{z,\overline{s}} =  TRD_z$. In such a situation, as in Case 1, one considers the complementary scenario defined as $s^C_z =1$ and $s^C_l = \overline{s_l}, l \neq z$, where $\overline{s_z}=0$ is replaced by $s^{C}_z =1$, necessarily implying $nrAct^*_{z,s^{C}} \leq nr^*_z$. This implies that \eqref{eq:XSP4N}-\eqref{eq:XSP9N} remains feasible whether $nrAct_{z}^{as-param} > nr^*_z$ or $nrAct_{z}^{as-param} \leq nr^*_z$.

As in Case 1, Lemma 1 in \cite{geoffrion} then ensures that there exists a solution where $nrAct^*_{z,\overline{s}} \leq nr^*_z$ for the initial scenario $s$ being considered, because \eqref{eq:XSP4N}-\eqref{eq:XSP9N} remains feasible for any $nrAct_{z}^{as-param}$ in the convex hull of $\{nrAct^*_{z,\overline{s}} > nr^*_z, nrAct^*_{z,s^C} \leq nr^*_z\}$. 

If $nrAct^*_{z,\overline{s}} < 0$, necessarily, $\overline{s}_z = 1$. One then considers the complementary scenario $s^C$ defined as $s^C_z =0$ and $s^C_l = \overline{s_l}, l \neq z$, where $\overline{s_z}=1$ is replaced by $s^{C}_z =0$, for which necessarily, $nrAct^*_{z,s^C} \geq 0$. With the same arguments as above, Lemma 1 in \cite{geoffrion} then ensures that $nrAct^*_{z,\overline{s}} \geq 0$ is also achievable for that scenario.\end{proof}

\textbf{Proof of Lemma \ref{lemma-shifting_total_demand_supply}}

\vspace{-0.4cm}

\begin{proof}
    Indeed, let us consider the two possible activations  $s_z\in \{0;1\}$ of the new total demand $TRD^\#_z = (TRD^*_z + K)$ in zone $z$ in any scenario $s$.
    
    \begin{itemize}
        \item For $s_z  = 1 $, $TRDAct^\#_{z,s} = TRD^\#_{z} = (TRD^*_z + K) $ and one just needs to consider accordingly the scenario-dependent supply $TRSAct^\#_{z,s} := (TRSAct^*_{z,s} + K)$, and for all other zones $l \neq z$,  $TRDAct^\#_{l,s} = TRDAct^*_{l,s}$, $TRSAct^\#_{l,s} = TRSAct^*_{l,s}$, which leaves all scenario-dependent net positions $nrAct_{z,s}$ unchanged, thereby maintaining feasibility for $XSP$.
        
        \item For  $s_z  = 0$, by definition of the scenario, $TRDAct^\#_{z,s} = 0 = TRDAct^*_{z,s}$, and one then just needs to consider the scenario-dependent supply $ TRSAct^\#_{z,s} := TRSAct^*_{z,s}$ which was associated with $TRDAct^*_{z,s}$ in the original case before shifting  both demand and supply values by $K$, while for all other zones $l \neq z$,  $TRDAct^\#_{l,s} = TRDAct^*_{l,s}$, $TRSAct^\#_{l,s} = TRSAct^*_{l,s}$. This again leaves all scenario-dependent net positions $nrAct_{z,s}$ unchanged, in the sense that $nrAct^\#_{z,s} = nrAct^*_{z,s}$, thereby maintaining feasibility for $XSP$.
    \end{itemize} \end{proof}

\vspace{-0.5cm}

\textbf{Proof of Proposition \ref{proposition-initsp-sp}}
\vspace{-0.4cm}

\begin{proof}
    
Showing that  $Proj_{dr, pr, nr}(InitSP) \subseteq Proj_{dr, pr, nr}(SP)$ is direct. Given a point \\ $(dr, pr, nr, nrAct, drAct, prAct) \in InitSP$, we need to build a point of $SP$ having same first-stage decision values $(dr, pr, nr)$. This requires finding corresponding values for $TRD_z$, $TRS_z$, $TRDAct$, $TRSAct$, to obtain a point in $SP$. 

For that purpose, the equations \eqref{eq:TRDSupply1}-\eqref{eq:TRDSupply4}, can be used, which besides defining adequate values for $TRD_z$, $TRS_z$ via \eqref{eq:TRDSupply1}-\eqref{eq:TRDSupply2}, also define via \eqref{eq:TRDSupply3}-\eqref{eq:TRDSupply4} values for second-stage decisions $TRDAct_{z,s}$, $TRSAct_{z,s}$ \emph{for each scenario $s \in [0;1]^{|RDB|}$, i.e. for each scenario of activation of the bids $dr$}.

We can then directly obtain values $TRDAct_{z,s}$, $TRSAct_{z,s}$ for activation scenarios of $TRD_z$ as defined in the formulation of $XSP$ (where in the second stage, $TRD_z$ is either fully activated or not activated at all), by considering a subset of those scenarios of $InitSP$ defined in terms of the $dr$ activations. Indeed, the extreme activation scenarios of $TRD_z$ correspond to the scenarios where all bids $dr_r$ of a same zone $z$ are activated, in which case $TRDAct_{z,s} = \sum_{r | z(r) = z} drAct_{r,s}$ equals $TRD_z$ (full activation), or where none of them is activated, in which case $TRDAct_{z,s} = \sum_{r | z(r) = z} drAct_{r,s}$ equals 0 (no activation).

\begin{align}
    &TRD_z = \sum_{r | z(r) = z} dr_r \taglabel{TRDSupply1} \\
    &TRS_{z}  = \sum_{g | z(g) = z} pr_g \taglabel{TRDSupply2}
\end{align}

And $\forall s=(s_1, ..., s_r, ..., s_{|RDB|}) \in \{0;1\}^{|RDB|}$
\begin{align}
    &TRDAct_{z,s} = \sum_{r | z(r) = z} drAct_{r,s} \taglabel{TRDSupply3} \\
    &TRSAct_{z,s} = \sum_{g | z(g) = z} prAct_{g,s} \taglabel{TRDSupply4}
\end{align}

Showing that  $Proj_{dr, pr, nr}(SP) \subseteq Proj_{dr, pr, nr}(InitSP)$ requires more attention. In particular, we need to show that the introduction of the auxiliary variables $TRD$, $TRS$, $TRDAct$, $TRSAct$ does not make feasible net injection variable values $nr$ that would not be admissible under $InitSP$.

We would like to construct a point of $InitSP$ having the same values $dr, pr, nr$ as the original point of $SP$. The values of $pr$, $dr$ and $nr$ are already determined in the $SP$ solution but we need to set values for the variables $drAct$, $prAct$ (which in turn fully determine the $nrAct$) by leveraging the values of $TRDAct$, $TRSAct$.

\emph{To proceed, we first address the fact that in the $SP$ solution, $TRD_z, TRS_{z}$ may not exactly match the sum of the variables $dr, pr$ as in \eqref{eq:TRDSupply1}-\eqref{eq:TRDSupply2}, since this is not explicitly enforced by $SP$}. By leveraging the netting Lemma \ref{lemma-netting_total_demand_supply} and the shifting Lemma \ref{lemma-shifting_total_demand_supply}, we redefine $TRD$ and $TRS$ (first-stage decisions in $SP$) to make them satisfy the equations \eqref{eq:TRDSupply1}-\eqref{eq:TRDSupply2}, without invalidating the satisfaction of the $SP$ constraints.  Let us consider a feasible point of $SP$. By Lemma \ref{lemma-netting_total_demand_supply}, we can assign to all the auxiliary variables $TRD_z, TRS_z$ new netted values $TRD_z^\#$, $TRS_z^\#$ according to the equations (\ref{eq:XSP-Netting1})-(\ref{eq:XSP-Netting2}) while maintaining feasibility for $SP$ (i.e. there exist corresponding values $TRDAct_{z,s}^\#$, $TRSAct_{z,s}^\#$, $nrAct_{z,s}^\#$ for the scenario-dependent variables).  By Lemma \ref{lemma-shifting_total_demand_supply}, we can then shift the values of these auxiliary variables to make them correspond respectively to $\sum_{r | z(r) = z} dr_r$ and $\sum_{g | z(g) = z} pr_g$, all of this leaving the values of $dr, pr, nr$ unchanged.

Now consider a scenario $s=(s_1, ..., s_r, ..., s_{|RDB|}) \in \{0;1\}^{|RDB|}$, for which we need to determine the $drAct$, $prAct, nrAct$ to obtain a point of $InitSP$ having the same first-stage decision values $dr, pr, nr$ as a given point of $SP$. Firstly, this scenario fully determines the values $drAct$ given the constraints \eqref{eq:InitSP10} requiring that $drAct_{r,s} = s_r dr_r, \forall r \in RDB$.

It remains to show that one can then define values $prAct$ satisfying \eqref{eq:InitSP11}, such that the net positions $nrAct$ implied via \eqref{eq:InitSP12} satisfy the constraints \eqref{eq:InitSP13}-\eqref{eq:InitSP14}.

Based on the values $drAct_{r,s} = s_r dr_r$, values $TRDAct_{z,s}$ can be set via  equations \eqref{eq:TRDSupply3}, which may correspond to a non-extreme activation scenario of the $TRD_z$, i.e. which may be such that for some $z$, $0 < TRDAct_{z,s} < TRD_z$. It could also correspond to an extreme scenario where for each $z$,  $TRDAct_z=TRD_z$ or $TRDAct_z=0$.

(a) If it corresponds to an extreme activation scenario of $TRD$, by definition of $XSP$, we know that there exist values $TRSAct, nrAct$ such that all constraints of $XSP$ are satisfied. Therefore, any set of values $prAct_{g,s} \leq pr_g$ that satisfy \eqref{eq:TRDSupply4} will ensure that \eqref{eq:InitSP12}, \eqref{eq:InitSP13}-\eqref{eq:InitSP14} are satisfied. Such a set of values can always be determined, since $TRSAct_{z,s} \leq TRS_z$ and since the $TRS_z$ satisfy \eqref{eq:TRDSupply2} after applying if needed, as described above, Lemma \ref{lemma-netting_total_demand_supply} and Lemma \ref{lemma-shifting_total_demand_supply}.

(b) If it corresponds to a non-extreme activation scenario of $TRD$, the same arguments still apply, given Proposition \ref{prop:finite-set-scenarios} ensuring that adequate $TRSAct, nrAct$ also exist as responses to non-extreme activation scenarios of the $TRD$, if they exist for extreme activation scenarios. \end{proof}

\vspace{-0.2cm}

\textbf{Proof of Theorem \ref{theorem-key-inner-approximation-theorem}}
    \vspace{-0.2cm}
\begin{proof}

    We need to show that if $nr$ satisfies the constraints \eqref{eq:INAP-1}-\eqref{eq:INAP-3} with the $W^i_k $ as defined in Definition \ref{def:secure-czc}, then $nr \in  Proj_{nr}(XSP)$ (there is a point of $XSP$ having the same first-stage net injection variables).

    This requires that there exist variable values $TRD$, $TRS$, $TRDAct$, $TRSAct$, $nrAct$ such that all constraints \eqref{eq:XSP3}-\eqref{eq:XSP9} are satisfied, in particular constraints \eqref{eq:XSP4}-\eqref{eq:XSP9} for each scenario $s$.

    Let us check that such values are directly obtained from the values $TRD^i$, $TRS^i$, $TRDAct^i$, $TRSAct^i$, $nrAct^i$ whose existence follows from the fact that $nr^i  \in Proj_{nr}(XSP[F_{k}^{\max} \rightarrow W^i_k])$, since the $W_k^i$ are secure capacities for the net position vectors $nr^i$ (see Definition \ref{def:secure-czc}).

    Given $nr^i  \in Proj_{nr}(XSP[F_{k}^{\max} \rightarrow W^i_k])$ and Definition \ref{def:secure-czc}, there exist variable values $TRD^i$, $TRS^i$ and their second-stage counterparts $TRDAct^i$, $TRSAct^i$, $nrAct^i$  such that, after replacing $F_{k}^{\max}$ by the $W_k^i$ in the right-hand sides of \eqref{eq:XSP8}, constraints \eqref{eq:XSP3}-\eqref{eq:XSP9} are all satisfied.

    Let us define the values $TRD$, $TRS$, $TRDAct$, $TRSAct$ as follows, by summing up the corresponding values corresponding to the $nr^i$ (where the assignments are compactly written in vector notation):

    \vspace{-0.6cm}
\begin{align}
& TRD := \sum_{i= 1, ..., T} TRD^i f_i \\
& TRS := \sum_{i= 1, ..., T} TRS^i f_i \\
& TRDAct := \sum_{i= 1, ..., T} TRDAct^i f_i \\
& TRSAct := \sum_{i= 1, ..., T} TRSAct^i f_i
\end{align}

Given how values $TRD$, $TRS$, $TRDAct$, $TRSAct$ have been set above, and given that the values $nrAct^i$, $TRDAct^i$, $TRSAct^i$ satisfy \eqref{eq:XSP3}-\eqref{eq:XSP9} once $F_{k}^{\max}$ are replaced by the $W_k^i$ in the right-hand sides of \eqref{eq:XSP8}, one directly checks the validity of \eqref{eq:SINAP-1}-\eqref{eq:SINAP-2}, with \eqref{eq:SINAP-2} implying \eqref{eq:XSP8} for all scenarios. Likewise, a direct verification shows that the values $TRD$, $TRS$, $TRDAct$, $TRSAct$, $nrAct$ set above are such that \eqref{eq:XSP3}-\eqref{eq:XSP9} are all satisfied, thereby proving that $nr \in Proj_{nr}(XSP)$, since satisfaction of the remaining constraints \eqref{eq:XSP1}-\eqref{eq:XSP2} can straightforwardly be checked.

\vspace{-0.6cm}
\begin{align}
& \begin{pmatrix}
nrAct_{1,s} \\
\vdots \\
nrAct_{z,s} \\
\vdots \\
nrAct_{|N|,s}
\end{pmatrix}
:=\sum_{i= 1, ..., T} \begin{pmatrix} 
nrAct^i_{1,s} \\
\vdots \\
nrAct^i_{z,s} \\
\vdots \\
nrAct^i_{|N|,s}
\end{pmatrix} f_i \taglabel{SINAP-1} \\
& \sum_z PTDF_{k,z} nrAct_{z,s} = \sum_z \sum_i  PTDF_{k,z} nrAct^i_{z,s} f_i \leq  \sum_i W^{i}_k \ f_i \leq \widetilde{F_k^{max}} & \forall k \in CNE \taglabel{SINAP-2}
\end{align}

\end{proof}

\vspace{-1cm}

\bibliographystyle{apalike-ejor}

\bibliography{template}

\newcount\savedpage
\savedpage=\value{page}

\clearpage

\begingroup
  \pagestyle{empty}

  

\section*{Supplementary material}

# Supplementary material - Models and Algorithms for Reserve Deliverability in Cross-Zonal Balancing Capacity Markets

Mehdi Madani[a,*], Zejun Ruan[b] and Anthony Papavasiliou[b]

[a]N-SIDE, Courbevoie 13, Louvain-la-Neuve, 1348, , Belgium

[b]National Technical University of Athens, Iroon Polytechneiou 9, Zografou, Athens, 157 72, , Greece

A B S T R A C T

We shortly discuss here how approaches developed in this paper extend to contexts where both upward and downward reserves are in scope, by showing how related models in the professional literature can be derived as special cases of Theorem 1.

## 1. Extensions to Both Upward and Downward Reserves

When both upward and downward reserves are in scope, the main point of attention is the following: the way power balance conditions per activation scenario are written should allow demand activations of upward and downward reserves to offset each other if the network allows for this, instead of requiring that they are always met by activations of supply orders. This is what European Transmission System Operators call “imbalance netting", and which is implicitly used to derive the model in Papavasiliou & Avila (2024) on pages 39-40, see constraints E.3 and E.8. We further discuss below via an example why the models in Consentec GmbH & IAEW at RWTH Aachen University (2025), for instance constraints (4.43) and (4.44) on page 42, may not fully ensure reserve deliverability.

Focusing on the core new concept introduced in this paper, we show how the general principle in Theorem 1 can be used to derive a new inner approximation model which we show to be equivalent to the model in the study in Papavasiliou & Avila (2024). The equivalent new model has more bilateral exchange variables and fewer network constraints. This shows the flexibility of the construction in Theorem 1 enabling to derive many interesting special cases. Note that to the best of our knowledge, no formal proof or detailed discussion of the correctness of the model used in Papavasiliou & Avila (2024) is currently available elsewhere.

A comprehensive technical treatment developing in extenso the model adaptations would require another paper. Therefore, we limit ourselves here to outlining through an example the mechanics for applying Theorem 1 - and typically the calculation of the secure cross-zonal capacity parameters $W_i^k$ (see Definition 1) - in this new context with both upward and downward reserves.

Compared to the base case in Section 3, an exchange pattern indexed by $i$ being combined in (INAP-1)-(INAP-3) will now involve net positions $nr^i_{z,up}$ and $nr^i_{z,down}$ for both upward and downward reserve net positions.

Like in Section 4, let us consider only *bilateral* exchange patterns, although more general exchanges like in Section 5 could also be considered in theory. To simplify notation and ease exposition while providing an example general enough, we assume that we have four nodes $A, B, C, D$ and look at bilateral exchanges of upward and downward reserves between $B$ and $C$ generating an energy flow from $B$ to $C$ if they are activated.

Let us consider the following three exchange patterns, that we will combine like in (INAP-1)-(INAP-3), where the vector components are in the order

*Corresponding author
mma@n-side.com (M. Madani); zruan@mail.ntua.gr (Z. Ruan); papavasiliou@mail.ntua.gr (A. Papavasiliou)

$$(nr^i_{A,up}, nr^i_{B,up}, nr^i_{C,up}, nr^i_{D,up}, nr^i_{A,down}, nr^i_{B,down}, nr^i_{C,down}, nr^i_{D,down})$$

- Exchange $i = 1$: upward reserve from $B$ to $C$: $(0, 1, -1, 0, 0, 0, 0, 0)$
- Exchange $i = 2$: downward reserve from $C$ to $B$: $(0, 0, 0, 0, 0, -1, +1, 0)$
- Exchange $i = 3$: both upward reserve from $B$ to $C$ and downward reserve from $C$ to $B$: $(0, 1, -1, 0, 0, -1, +1, 0)$

After renaming the indices $i = 1$ by "$B \to C, up$", index $i = 2$ by "$C \to B, down$", and index $i = 3$ by "$B \to C, up, C \to B, down$", the balance conditions (INAP-1) in this case read:

$$
\begin{aligned}
nr_{B,up} &= f_{B\to C,up} + f_{B\to C,up,C\to B,down}\\
nr_{C,up} &= -f_{B\to C,up} - f_{B\to C,up,C\to B,down}\\
nr_{B,down} &= -f_{C\to B,down} - f_{B\to C,up,C\to B,down}\\
nr_{C,down} &= +f_{C\to B,down} + f_{B\to C,up,C\to B,down}
\end{aligned}
$$

Let us now determine secure cross-zonal capacities in the sense of Definition 1 appearing as the $W^i_k$ in (INAP-2).

1. $W_k^{B\to C,up} = max(0, PTDF_{k,B} - PTDF_{k,C})$, exactly as in Section 4 for bilateral exchanges of upward reserve only.
2. $W_k^{C\to B,down} = max(0, PTDF_{k,B} - PTDF_{k,C})$, (same as with the single upward product exchange which precedes, except that for a downward reserve product, an exchange from $C$ to $B$ corresponds to a power flow from $B$ to $C$).
3. The determination of $W_k^{B\to C,up,C\to B,down}$ (third exchange pattern considered above, corresponding to the initial index $i = 3$) is more interesting, as it involves the idea of "imbalance netting" well-known by Transmission System Operators, and implicitly used in the model in Papavasiliou & Avila (2024) that we recover below as a special case.

   We show that the maximum stress on the network element $k$ across all activation scenarios is again $max(0, PTDF_{k,B} - PTDF_{k,C})$ which are therefore necessary and sufficient secure cross-zonal capacities in the sense of Definition 1. Indeed, let us consider the following four extreme upward and downward demand activation scenarios:

   (a) The upward demand imported in $C$ is activated but not the downward demand imported in $B$. This case is similar to the case in (1), leading to a value $PTDF_{k,B} - PTDF_{k,C}$
   (b) The upward demand imported in $C$ is not activated but the downward demand imported in $B$ is activated. This case is similar to the case in (2), leading to a value $PTDF_{k,B} - PTDF_{k,C}$
   (c) If both the upward demand in $C$ and the downward demand in $B$ are activated, then the best response is to let them offset each other, leading to an export of energy of 1 MW from $B$ to $C$, again leading to a stress on $k$ of $PTDF_{k,B} - PTDF_{k,C}$
   (d) Finally, if none are activated, the impact on the network is 0.

   Therefore, the $W_k^{B\to C,up,C\to B,down}$ set equal to $max(0, PTDF_{k,B} - PTDF_{k,C})$ are also secure capacities in the sense of Definition 1 for the exchange pattern $f_{B\to C,up,C\to B,down}$.

We therefore obtain the following flow-based constraint for the critical network element $k$, ensuring reserve deliverability:

$$max(0, PTDF_{k,B} - PTDF_{k,C})(f_{B\to C,up} + f_{C\to B,down} + f_{B\to C,up,C\to B,down}) \le RAM_k \tag{1}$$

or more generally if all pairs of zones $z_1, z_2$ are considered and not just $B, C$:

$$\sum_{z_1\neq z_2} max(0, PTDF_{k,z_1} - PTDF_{k,z_2})(f_{z_1\to z_2,up} + f_{z_2\to z_1,down} + f_{z_1\to z_2,up,z_2\to z_1,down}) \le RAM_k \tag{UD-1}$$

with balance conditions in this general case:

$$nr_{z_1,up} = \sum_{z_2 \neq z_1} f_{z_1\to z_2,up} - f_{z_2\to z_1,up} + f_{z_1\to z_2,up,z_2\to z_1,down} - f_{z_2\to z_1,up,z_1\to z_2,down} \tag{UD-2}$$

$$nr_{z_1,down} = \sum_{z_2 \neq z_1} f_{z_1\to z_2,down} - f_{z_2\to z_1,down} - f_{z_1\to z_2,up,z_2\to z_1,down} + f_{z_2\to z_1,up,z_1\to z_2,down} \tag{UD-3}$$

$$f \geq 0 \tag{UD-4}$$

We now prove that the inner approximation given by (UD-1)-(UD-4) is equivalent to the more standard inner approximation in Papavasiliou & Avila (2024) , compare (BaseUD-1) to constraints E.3, E.8 on pages 39-40 in Papavasiliou & Avila (2024):

**Proposition S1.** *The set of net position vectors* $(nr_{z,up}, nr_{z,down})$ *achievable under constraints* (UD-1)-(UD-4) *exactly corresponds to the net position vectors achievable with the constraints below* (BaseUD-1)-(BaseUD-4)*, similar to the ones in* Papavasiliou & Avila (2024)*, pages 39-40 (constraints E.3 and E.8).*

$$\sum_{z_1 \neq z_2} max(0, PTDF_{k,z_1} - PTDF_{k,z_2}) max(F_{z_1\to z_2,up}, F_{z_2\to z_1,down}) \leq RAM_k \tag{BaseUD-1}$$

$$nr_{z_1,up} = \sum_{z_2 \neq z_1} F_{z_1\to z_2,up} - F_{z_2\to z_1,up} \tag{BaseUD-2}$$

$$nr_{z_1,down} = \sum_{z_2 \neq z_1} F_{z_1\to z_2,down} - F_{z_2\to z_1,down} \tag{BaseUD-3}$$

$$F \geq 0 \tag{BaseUD-4}$$

Before moving to the formal proof, let us illustrate its basic logic through a numerical example.

The idea is that flows $F_{z_1\to z_2,up}, F_{z_2\to z_1,down}$ in (BaseUD-2)-(BaseUD-4) embed both joint upward/downward flows $f_{z_1\to z_2,up,z_2\to z_1,down}$, and the upward-only or downward-only flows $f_{z_1\to z_2,up}, f_{z_2\to z_1,down}$.

For the direction $BaseUD$ to $UD$, suppose that $F_{z_1\to z_2,up} = 30$ and $F_{z_2\to z_1,down} = 20$ in (BaseUD-2)-(BaseUD-3). To define a point of $UD$ having the same $nr$ net position values, we then simply define a flow $f_{z_1\to z_2,up,z_2\to z_1,down} = 20$ (their common part, i.e. the minimum of both), and then set accordingly the remaining flows $f_{z_1\to z_2,up} = 10, f_{z_2\to z_1,down} = 0$ to maintain the same global cross-zonal exchanges.

These flows $f$ trivially satisfy (UD-1) since $f_{z_2\to z_1,down} = 0$, the key observation in the general case being that this procedure always leads to having either $f_{z_2\to z_1,down} = 0$ or $f_{z_1\to z_2,up} = 0$ (or both equal to 0), depending on whether $F_{z_1\to z_2,up} \geq F_{z_2\to z_1,down}$ or $F_{z_2\to z_1,down} \geq F_{z_1\to z_2,up}$.

In the other direction from $UD$ to $BaseUD$, suppose we have the values $f_{z_1\to z_2,up,z_2\to z_1,down} = 15$, $f_{z_1\to z_2,up} = 15$ and $f_{z_2\to z_1,down} = 5$. We can then define $F_{z_1\to z_2,up} := f_{z_1\to z_2,up} + f_{z_1\to z_2,up,z_2\to z_1,down}$ and $F_{z_2\to z_1,down} := f_{z_2\to z_1,down} + f_{z_1\to z_2,up,z_2\to z_1,down}$.

A direct inspection shows that they satisfy (BaseUD-1), since (UD-1) holds for the $f$ variables.

*Proof.* Part 1 from $BaseUD$ to $UD$. Consider a solution to (BaseUD-1)-(BaseUD-4). Let us define a solution to (UD-1)-(UD-4) having the same $nr$ values as follows:

$$f_{z_1\to z_2,up,z_2\to z_1,down} := min(F_{z_1\to z_2,up}, F_{z_2\to z_1,down})$$

$$f_{z_1\to z_2,up} := F_{z_1\to z_2,up} - min(F_{z_1\to z_2,up}, F_{z_2\to z_1,down})$$

$$f_{z_2\to z_1,down} := F_{z_2\to z_1,down} - min(F_{z_1\to z_2,up}, F_{z_2\to z_1,down})$$

Necessarily, at least one of the two flows $f_{z_1 \to z_2,up}, f_{z_2 \to z_1,down}$ defined in this way must be equal to 0 since $F_{z_1 \to z_2,up}$ or $F_{z_2 \to z_1,down}$, or both, reach $min(F_{z_1 \to z_2,up}, F_{z_2 \to z_1,down})$: It is then a direct check to verify that these $f$ values satisfy (UD-1)-(UD-4).

Part 2 from $UD$ to $BaseUD$. Consider a point of (UD-1)-(UD-4) and let us define a point of (BaseUD-1)-(BaseUD-4) having the same $nr$ values.

We can then define $F_{z_1 \to z_2,up} := f_{z_1 \to z_2,up} + f_{z_1 \to z_2,up,z_2 \to z_1,down}$ and $F_{z_2 \to z_1,down} := f_{z_2 \to z_1,down} + f_{z_1 \to z_2,up,z_2 \to z_1,down}$: A direct inspection shows that they satisfy (BaseUD-1), since (UD-1) holds for the $f$ variables. □

Let us now come back to the models in Consentec GmbH & IAEW at RWTH Aachen University (2025), for instance constraints (4.43) and (4.44) on page 42, which may not fully ensure reserve deliverability.

Inspection of constraint (4.43) reveals that instead of considering constraints of the form of (BaseUD-1), the models rather apply the $max$ operator in the following way

$$max\left[\sum_{z_1 \neq z_2} max(0, PTDF_{k,z_1} - PTDF_{k,z_2}) F_{z_1 \to z_2,up}, \sum_{z_1 \neq z_2} max(0, PTDF_{k,z_1} - PTDF_{k,z_2}) F_{z_2 \to z_1,down}\right] \leq RAM_k \tag{C-UD-1}$$

which can be rewritten equivalently as

$$\sum_{z_1 \neq z_2} max(0, PTDF_{k,z_1} - PTDF_{k,z_2}) F_{z_1 \to z_2,up} \leq RAM_k \tag{C-UD-1a}$$

$$\sum_{z_1 \neq z_2} max(0, PTDF_{k,z_1} - PTDF_{k,z_2}) F_{z_2 \to z_1,down} \leq RAM_k \tag{C-UD-1b}$$

We briefly discuss why this modeling approach, sometimes used by practitioners in the presence of upward and downward reserves, may not satisfy the reserve deliverability requirement. This model $CUD$ is similar to making appear upward and downward bilateral exchanges in separate duplicate flow-based constraints, see (C-UD-1a)-(C-UD-1b) (or alternatively requires that the maximum of both left-hand sides in (C-UD-1a)-(C-UD-1b) does not exceed capacities $RAM_k$, as in (C-UD-1)). It implicitly allocates the same cross-zonal capacity to both upward and downward exchanges, which may lead to insecure activations of reserves as the example below shows. The above models $UD$ or $BaseUD$ instead fully ensure reserve deliverability.

The example consists of the PTDF constraint $0.5nr_A + nr_C \leq 0.5$.

Let us consider a situation where $A$ exports upward capacity to $B$ and exports downward capacity to $C$.

These exchanges can be decomposed into the following zone-to-zone bilateral exchanges: $F_{A \to B,up}$ and $F_{A \to C,down}$. In the constraints (C-UD-1a)-(C-UD-1b) above, these bilateral exchanges have respectively the coefficients $max(0, PTDF_A - PTDF_B) = 0.5$ and $max(0, PTDF_C - PTDF_A) = 0.5$.

Therefore, having $A$ exporting 1 MW of upward balancing capacity to $B$ (corresponding to $F_{A \to B,up} = 1$) and also exporting 1 MW of downward capacity to $C$ (corresponding to $F_{A \to C,down} = 1$) will appear feasible for the model $CUD$, as (C-UD-1a)-(C-UD-1b) will be separately satisfied if the RAM (right-hand side of the flow-based constraint) is 0.5.

However, if both the upward reserve in $B$ and the downward reserve in $C$ need to be activated, the impact on the flow-based constraint will be equal to $PTDF_C - PTDF_B = 1\ MW$ which violates the available capacity. In contrast, the models $UD$ or $BaseUD$ would require 1 MW to allow for the proposed exchanges above, and would not allow them if the capacity of the flow-based constraint is 0.5 as above.



\endgroup

\setcounter{page}{\savedpage}

\end{document}